\documentclass[11pt]{amsart}

\usepackage[english]{babel}
\usepackage{graphicx,amsmath,amssymb,color, enumerate}
\numberwithin{equation}{section}
\usepackage{enumerate}
 \usepackage{xcolor}
 \definecolor{brickred}{rgb}{0.8, 0.25, 0.33}
\definecolor{blue(ryb)}{rgb}{0.01, 0.28, 1.0}
\definecolor{brandeisblue}{rgb}{0.0, 0.44, 1.0}
\definecolor{ceruleanblue}{rgb}{0.16, 0.32, 0.75}
\definecolor{cobalt}{rgb}{0.0, 0.28, 0.67}
\definecolor{coolblack}{rgb}{0.0, 0.18, 0.39}
\definecolor{darkblue}{rgb}{0.0, 0.0, 0.55}
\usepackage[hidelinks]{hyperref}
\hypersetup{
    colorlinks,
    citecolor=darkblue,
    filecolor=black,
    linkcolor=darkblue,
    urlcolor=black
}

\usepackage{comment}

\newtheorem{theorem}{Theorem}[section]

\newtheorem{proposition}[theorem]{Proposition}
\newtheorem{corollary}[theorem]{Corollary}

\theoremstyle{remark}

\newtheorem{remark}[theorem]{Remark}

\newtheorem{definition}[theorem]{Definition\rm}

\newcommand{\ee}{\mathrm{e}}
\newcommand{\dd}{\mathrm{d}}

\newcommand{\cO}{\mathcal{O}}

\newcommand{\N}{\mathbb{N}}
\newcommand{\Z}{\mathbb{Z}}
\newcommand{\R}{\mathbb{R}}

\newcommand{\Ab}{\mathbf{A}}
\newcommand{\Fb}{\mathbf{F}}

\newcommand{\cL}{\mathcal L}

\newcommand{\nb}{\mathbf{n}}

\newcommand{\ii}{\,\mathrm{i}}

\DeclareMathOperator{\curl}{curl}

\newcommand{\half}{\frac{1}{2}}

\title[Flux effects on magnetic eigenvalues]{Flux effects on Magnetic Laplace and Steklov eigenvalues in the exterior of a disk}

\author[B. Helffer]{Bernard Helffer}
\address[B. Helffer]{Laboratoire de Math\'ematiques Jean Leray, CNRS, Nantes Universit\'e,
	44000 Nantes, France.}
\email{Bernard.Helffer@univ-nantes.fr}
\author[A. Kachmar]{Ayman Kachmar}
\address[A. Kachmar]{Department of Mathematics and PDE Research Unit–Center for Advanced Mathematical Sciences, American University of Beirut, P.O.Box 11-
0236, Riad El-Solh, Beirut 1107 2020, Lebanon.}
\email{ak292@aub.edu.lb}
\author[F. Nicoleau]{Fran\c{c}ois Nicoleau}
\address[F. Nicoleau]{Laboratoire de Math\'ematiques Jean Leray, CNRS, Nantes Universit\'e,
	44000 Nantes, France.}
\email{francois.nicoleau@univ-nantes.fr}

\begin{document}
\begin{abstract}
We derive a three-term asymptotic expansion for the lowest eigenvalue of the magnetic Laplace and Steklov operators in the exterior of the unit disk in the strong magnetic field limit. This improves recent results of Helffer--Nicoleau (2025) based on special function asymptotics, and extends earlier works by Fournais--Helffer (2006), Kachmar (2006), and R. Fahs, L. Treust, N. Raymond, S. V\~u  Ng\d{o}c (2024). Notably, our analysis reveals how the third term encodes the dependence on the magnetic flux. Finally, we investigate the weak magnetic field limit and establish the flux dependence in the asymptotics of Kachmar--Lotoreichik--Sundqvist (2025).
 \end{abstract}
\maketitle

\section{Introduction}

Motivated by the mathematical theory of type-II superconductivity, the strong field limit for the magnetic Laplacian has been extensively studied (see  \cite{FH-b} and references therein). For the magnetic Steklov eigenvalue problem---inspired by \cite{CGHP}---investigations began with the disk case \cite{HN1} and were later extended to general smooth two-dimensional domains \cite{HKN}. Recently, \cite{Shen} established leading-order asymptotics in higher dimensions for both operators.

While the weak magnetic field limit reduces to a regular perturbation for bounded domains, exterior domains exhibit singular behavior. This was recently studied for the magnetic Laplacian in \cite{KLS} and for the magnetic Steklov problem in \cite{HN}. An alternative Steklov-specific regularization via a weak scalar potential has been studied in  \cite{CG}.

In this work, we study the influence of the magnetic flux on the lowest eigenvalue of the magnetic Laplace and Steklov operators in the exterior of the unit disk,
\[\Omega=\{x\in\R^2\colon |x|>1\},\]
addressing both strong and weak magnetic field regimes. For the strong field limit, we derive three-term asymptotic expansions for both the magnetic Laplacian (Theorem~\ref{thm:main1}) and the magnetic Steklov problem (Corollary~\ref{corol:main1}). In the weak field regime, we prove accurate asymptotics for the Neumann magnetic Laplacian (Theorem~\ref{thm:gs-weakfield}).

We account for the flux effects through an additional Aharonov--Bohm potential that naturally arises as a topological effect from the non-simply connected nature of the exterior domain. Our asymptotics capture the dependence on this Aharonov--Bohm potential in addition to the magnetic field strength. In the particular regime of weak magnetic field, we show that the flux effects persist despite the negligible magnetic field, in agreement with the Aharonov--Bohm effect. The table below provides a roadmap to our findings

\begin{table}[h]
\centering
\begin{tabular}{lll}
&&\textbf{Flux Effects}\\
\hline\\
\textbf{Vector potential} & Magnetic field $b$ & Aharonov-Bohm  parameter\\
&Flux $\Phi$& $\nu=\Phi-b/2$\medskip\\

\textbf{Magnetic Laplace} & Lowest eigenvalue & Theorems~\ref{thm:main1} and \ref{thm:gs-weakfield}\\
&\emph{Large} $b$&3rd term in the asymptotics\\
&\emph{Small $b$}&2nd term in the asymptotics\medskip\\
\textbf{Magnetic Steklov}&Lowest eigenvalue& Corollary~\ref{corol:main1}\\
&\emph{Large $b$}& 3rd term in the asymptotics\\
&&$e_0$-sequence\smallskip\\
\hline
\end{tabular}
\end{table}
\subsection*{Problem formulation}~\\We consider a vector potential \(\mathbf{F} \colon \Omega \to \mathbb{R}^2\)  of class $C^1$ with constant curl,
\[\curl \Fb = b>0\,,\]
 and consider the eigenvalues 
\[\begin{gathered}
\mu(\Fb,\beta)=\inf_{u\in H^1(\Omega)}\frac{\|(-\ii\nabla-\Fb)u\|^2_{\Omega}+\beta\|u\|_{\Gamma}^2}{\|u\|^2_{\Omega}}\,,\\
    \lambda(\Fb)=\inf_{u\in H^1(\Omega)}\frac{\|(-\ii\nabla-\Fb)u\|^2_{\Omega}}{\|u\|^2_{\Gamma}}\,,
\end{gathered}\]
where $\|\cdot\|_{\Omega}$ and $\|\cdot\|_{\Gamma}$ denote the $L^2$-norms on $\Omega$ and $\Gamma:=\partial\Omega$ respectively, and $\beta\in\R$ is a given parameter. 

The quantity $\mu(\Fb,\beta)$ corresponds to the lowest eigenvalue of the magnetic Laplacian,
\begin{subequations}\label{defHam}
\begin{equation}
\mathcal L=(-\ii\nabla-\Fb)^2=(-\ii\nabla-\Fb)\cdot(-\ii\nabla-\Fb) \mbox{ 
in } \Omega\,,
\end{equation}
 subject to the Robin boundary condition
\begin{equation}
 \mathbf{n}\cdot(\nabla-\ii \Fb)u=\beta u\quad\mbox{on }\Gamma,
 \end{equation}
 \end{subequations}
where $\mathbf n$ is the unit normal vector to $\Gamma$ pointing inward $\Omega$\,.

     A slight modification of the argument in \cite[Theorem~1.1]{GKS} yields that the essential spectrum of the Robin magnetic Laplacian   is the same as the essential spectrum of the Landau Hamiltonian with Aharonov-Bohm solenoid\footnote{That is the operator $\mathcal H=(-\ii\nabla-\Fb)^2$ in $\R^2$ with $\Fb$ as in \eqref{eq:def-F}; when $\nu=0$, we recover the Landau Hamiltonian.}. Thanks to \cite{ESV}, this consists of the Landau levels $b,3b,5b,\cdots$.

Similarly, $\lambda(\Fb)$ is the lowest eigenvalue of the magnetic Steklov problem, characterized by the existence of a non-zero function $u$ satisfying
\begin{equation}\label{Stecklov}
\begin{cases}
    (-\ii\nabla-\Fb)^2u=0&\mbox{ on }\Omega\,,\\
    \nb\cdot (\nabla-\ii \Fb)u=-\lambda u&\mbox{ on }\Gamma.
\end{cases} 
\end{equation}
The minus sign in front of $\lambda$ arises from the convention, commonly used in the standard Steklov problem, of taking the outward unit normal vector to the domain.

Beyond its intrinsic interest, the magnetic Laplacian with Robin boundary condition provides key insights into the magnetic Steklov eigenvalue $\lambda(\Fb)$. This connection arises from the relation
\begin{equation}\label{eq:Robin-to-Steklov}
\mu(\Fb,\beta)=0\quad\mbox{ if and only if }\beta=-\lambda(\Fb)\,. 
\end{equation}
Due to unitary transformations, $\mu(\Fb,b)$ and $\lambda(\Fb)$ are uniquely determined by the magnetic field 
$\curl\Fb=b$ 
and since $\Omega$ is not simply connected,  by the  renormalized magnetic flux %
\[ \Phi:=\frac{1}{2\pi}\int_{\partial\Omega}\Fb\cdot \dd x \,.\]
We  fix the choice of the vector potential  as
\begin{equation}\label{eq:def-F}\Fb(x)=\frac{b}{2}(-x_2,x_1)+\frac{\nu}{|x|^2}(-x_2,x_1),\quad\nu=\Phi-\frac{b}{2}\,,
\end{equation}
and we choose the Robin  parameter $\beta$ as
\[ \beta=b^{1/2}\gamma\,,\quad \gamma\in\R\,.\]
Being dependent on the aforementioned parameters only, we denote the lowest magnetic Laplace and Steklov eigenvalues by
\begin{equation}\label{eq:def-ev}
\mu(b,\nu,\gamma):=\mu(\Fb,\beta),\quad \lambda(b,\nu):=\lambda(\Fb)\,.
\end{equation}
Furthermore, these eigenvalues are periodic with respect to $\nu$ with period~$1$ due to the unitary transformation 
\[ u\mapsto \ee^{-\ii\theta}u.\] Thus, there is no loss of generality\footnote{ While the lowest eigenvalue can be recovered by periodicity, the ground states are recovered via the unitary transformation, hence their radial symmetry depends on $\nu$.} in restricting the parameter $\nu$ to $(-1/2,1/2]$.  

\subsection*{Strong magnetic field}

Recently, for  the Steklov problem in the exterior of the disk, \cite[Theorem 1.5]{HN} established the asymptotics,
\begin{equation}\label{eq:asym-HN}\lambda(b,\nu)=\hat\alpha b^{1/2}+\frac{\hat\alpha^2+1}{3}+\cO(b^{-1/2})\quad\mbox{as $b\to+\infty$,}
\end{equation}
where $\hat\alpha$ is a positive universal constant independent of the flux parameter $\nu$.
As this formula does not display the flux effects dependent on $\nu$,  our aim is to capture those hidden terms. 

To that end, we establish an expansion with three terms for the lowest eigenvalue for the magnetic Laplacian, which involves spectral quantities dependent on the parameter $\gamma$, namely:
\begin{enumerate}
    \item $\Theta(\gamma)$,  an increasing and smooth function of $\gamma$, is the infimum of the lowest eigenvalues for the family of harmonic oscillators on the positive semi-axis
    \[\mathfrak h_0[\xi,\gamma]=-\frac{d^2}{dt^2}+(t-\xi)^2,\]
    with Robin boundary condition $u'(0)=\gamma u(0)$.
    \item $\xi(\gamma)=\sqrt{\Theta(\gamma)+\gamma^2}$.
    \item $\varphi_\gamma$ is a normalized ground state to $\mathfrak h_0[\xi(\gamma),\gamma]$.%
    \item $\mathcal C(\gamma)=\frac13\bigl(1-\gamma\xi(\gamma)\bigr)\Theta'(\gamma)$. 
\end{enumerate}
At this stage, we can describe the constant $\hat\alpha$ appearing in \eqref{eq:asym-HN} by  the equation $\Theta(-\hat\alpha)=0$,  which by the above implies  
\begin{equation}\label{eq:xi-C-alpha}
 \xi(-\hat\alpha)=\hat\alpha,\quad \mathcal C(-\hat\alpha)=\frac{1+\hat\alpha^2}{3}\Theta'(-\hat\alpha).
\end{equation}
It is proved in \cite{HN, HKN} that $\widehat{\alpha} = \frac{\alpha}{\sqrt{2}} \approx 0.5409019\ldots$, where $-\alpha$ denotes the unique negative zero of the parabolic cylinder function $D_{1/2}(z)$. Recall that the function $D_{1/2}(z)$ is the normalized solution of the differential equation
	\begin{equation}\label{ODEDnu0}
		w'' + \left(1 - \frac{z^2}{4}\right) w = 0\,,
	\end{equation}
	which decays to zero as $z \to +\infty$.

The three-term asymptotics for the lowest eigenvalue $\mu(b,\nu,\gamma)$  is  presented in the theorem below. This result generalizes\footnote{See also \cite{HL} for related recent results in the disk.} \cite[Theorem~2.5]{FH-cmp} and \cite[Theorem~5.3.1]{FH-b}, which addressed the case $\nu=0$ and $\gamma=0$\,.
\begin{theorem}[Flux dependence for magnetic Laplace]
    \label{thm:main1}
    Let $\gamma\in\R$ be fixed. There are constants 
    \[\mathcal C_0(\gamma)\in\R\quad\mbox{and}\quad \mathcal C_1(\gamma)\in\R\]
    such that, for  $\nu\in(-1/2,1/2]$, we have as $b\to+\infty\,$, 
    \begin{subequations}\label{eq:thm0}
    \begin{equation}
     \mu(b,\nu,\gamma)=\Theta(\gamma)b+\mathcal C(\gamma)b^{1/2}+\xi(\gamma)\Theta'(\gamma)
    \inf_{m\in\Z}\Delta_m(b,\nu,\gamma) +\mathcal O(b^{-1/2})\,,
    \end{equation}
     where
    \begin{equation}
    \Delta_m(b,\nu,\gamma)=\Bigl(m-\nu-\frac{b}{2}-b^{1/2}\xi(\gamma)-\mathcal C_0(\gamma)\Bigr)^2+\mathcal C_1(\gamma)\,.\end{equation}
    \end{subequations}
    Furthermore, \eqref{eq:thm0} holds uniformly with respect to $\nu\in(-1/2,1/2]$ and locally uniformly with respect to $\gamma\in\R$.
\end{theorem}

\begin{remark}\label{rem-thm}~
\begin{enumerate}[i)]
\item   Under Neumann boundary conditions ($\gamma=0$), we recover:
\begin{itemize}
\item  \cite[Theorem~5.3.1]{FH-b} for $\nu=0$\,;
\item a particular case of
\cite[Theorem~4.1]{FS} for $\nu\neq0$\,.
\end{itemize}
\item When $\gamma\neq0$ and $\nu=0$, the first two terms in \eqref{eq:thm0} are obtained  from either \cite{K-jmp} or from the general spectral reduction to an effective operator given  in \cite{FTRV}.
\item 
Thanks to \eqref{eq:def-F}, we observe that the oscillatory term in \eqref{eq:thm0} depends on the flux and the intensity of the magnetic field through the quantity
\[  \widehat \Delta(\Phi,b,\gamma):=\inf_{m\in\Z}\Bigl(m-\Phi-b^{1/2}\xi(\gamma)-\mathcal C_0(\gamma)\Bigr)^2\,.\]
\item
    The functions  $\mathbb R\ni \gamma \mapsto \mathcal C_0(\gamma)$ and $\mathbb R\ni \gamma \mapsto  \mathcal C_1(\gamma)$ are smooth, and can be expressed explicitly in terms of  the spectral parameters $\Theta(\gamma)$, $\xi(\gamma)$ and the  ground state $\varphi_\gamma$.
\end{enumerate}
\end{remark} 
As a consequence of Theorem~\ref{thm:main1}, we obtain an accurate asymptotics for the Steklov lowest eigenvalue which improves \eqref{eq:asym-HN}.
To state the new asymptotics, we borrow the notion of $e_0$-sequences from \cite{FM}.\medskip

\paragraph{\bf Definition.}
\emph{Let 
\begin{equation}\label{eq:defetabnu}
\eta(b,\nu):=\frac{b}{2}+b^{1/2}\hat\alpha+\frac{(\hat\alpha^2+1)(\Theta'(-\hat\alpha)-2\hat\alpha)}{6\hat\alpha}+\mathcal C_0(-\hat\alpha)+\nu\,,
\end{equation}
and  $e_0\in(-1/2,1/2]\,$. A sequence $(b_n)_{n\in \mathbb N}$ ($b_n>0$) is said to be an $e_0$-sequence if there is a sequence $(p_n)$ of integers such that\footnote{ In particular, if $\eta(b_n,\nu)-n\to e_0$, then $(b_n)$ is an $e_0$-sequence.} 
\[b_n\to+\infty\quad{\rm and}\quad \eta(b_n,\nu)-p_n\to e_0 \mbox{ as } n\rightarrow +\infty\,. \]}
The introduction of $e_0$-sequences is motivated by the oscillatory nature of the eigenvalues in the strong field limit. It provides a convenient way to characterize the optimal angular momentum: By restricting to an $e_0$-sequence, we effectively ``fix" the phase of these oscillations, allowing for a precise three-term expansion.

Note that, in terms of the flux and the intensity of the magnetic field, the quantity $\eta(b,\nu)$ is the same as \[\hat \eta(b,\Phi):=     \Phi +b^{1/2}\hat\alpha+\frac{(\hat\alpha^2+1)(\Theta'(-\hat\alpha)-2\hat\alpha)}{6\hat\alpha}+\mathcal C_0(-\hat\alpha).\]
\begin{corollary}[Flux dependence for magnetic Steklov]\label{corol:main1}
Let $\nu\in(-1/2,1/2]$. For a given $e_0\in(-1/2,1/2]$ and an $e_0$-sequence $(b_n)$, the lowest eigenvalue of the Steklov operator satisfies 
\[
    \lambda(b_n,\nu)=\hat\alpha b^{1/2}_n+\frac{\hat\alpha^2+1}{3}
    +\bigl(e_0^2 +k_0\bigr)\hat\alpha\, b^{-1/2}_n +\mathcal O(b^{-1}_n)\,,
\]
for some constant $k_0$.
\end{corollary}

The constant $k_0$ is universal and is given by (see Theorem~\ref{thm:main2} and \eqref{eq:def-K0}):
\[k_0=\mathcal C_1(-\hat\alpha)+\Bigl(\frac{\hat\alpha^2+1}{6}\Theta''(-\hat\alpha)-\mathcal C'(-\hat\alpha)\Bigr)\frac{\hat\alpha^2+1}{3\hat\alpha \Theta'(-\hat\alpha)}.\]
\noindent

\begin{remark}\label{rem:flux-effect}
While the  coefficients in the asymptotics of $\lambda(b_n,\nu)$ are independent of the  additional flux term $\nu$, the $e_0$-sequence $(b_n)$ depends on $\nu$.  In fact, we can take
\[b_n = 2n - 2^{\frac32} \hat \alpha \sqrt{n} + 2({\hat \alpha}^2 -A +e_0-\nu) + o(1) \mbox{ as } n \to + \infty\,,\]
where 
\[
A = \frac{(\hat\alpha^2+1)(\Theta'(-\hat\alpha)-2\hat\alpha)}{6\hat\alpha}+\mathcal C_0(-\hat\alpha)\,.
\]
This shows how  the additional flux term affects the convergence. 
\end{remark}
\subsection*{Weak magnetic field}
We investigate the limit $b\to0^+$ in the Neumann case ($\gamma=0$). For the case with no additional  flux term ($\nu=0$), 
 the lowest eigenvalue of the Neumann magnetic Laplacian in the exterior of the disk satisfies \cite[Theorem 1.1]{KLS}
\begin{equation}\label{eq:weakfield}
\mu(b,0,0)=b-b^2+o(b^2),
\end{equation}
and the corresponding ground states are not radially symmetric.  In polar coordinates, they have the  structure $f(r)\ee^{\ii\theta}$.\medskip

For the case of a general  additional flux $\nu$, we prove the following theorem.

\begin{theorem}[Aharonov-Bohm effect in vanishing fields]\label{thm:gs-weakfield}
Let $\nu\in(-1/2,1/2].$ There exists $b_0>0 $ such that, for $b\in(0,b_0)$, we have the following.
\begin{enumerate}[\rm 1.]
\item The ground state energy of the Neumann magnetic Laplacian is a simple 
eigenvalue and it satisfies as $b\to0^+$,
\[\mu(b,\nu,0)=\begin{cases}
\displaystyle    b-\frac{2^\nu}{\Gamma(1-\nu)}b^{2-\nu}+o(b^{2-\nu})&\mbox{ if $\nu\geq 0$\,,}\\
\displaystyle    b-\frac{2^{1+\nu}}{\Gamma(-\nu)}b^{1-\nu}+o(b^{1-\nu})&\mbox{ if $\nu< 0$\,.}
\end{cases} \]
\item  If $\nu\geq 0$, the ground state  of the Neumann magnetic Laplacian is not radially symmetric, whereas, if $\nu<0$, it is  radially symmetric.
\end{enumerate}
\end{theorem}
\begin{remark}\label{rem:weakfield*}
    The study of the non-Neumann case  can be subtle for the following reason.  In light of \eqref{eq:Robin-to-Steklov}, we have $\mu(\Fb,\beta)=0$. This corresponds to $\mu\bigl(b,\nu,\gamma(b,\nu)\bigr)=0$ with $\gamma=-b^{-1/2}\lambda(b,\nu)$.  By \cite{HN}, we have
    \begin{itemize}
    \item for $b$ sufficiently small, $\mu\bigl(b,\nu,\gamma(b,\nu)\bigr)$ is a simple eigenvalue with a radially symmetric ground state;
        \item the Robin parameter $\beta=b^{1/2}\gamma(b,\nu)$ satisfies as $b\to0^+$:
        \[
    \beta=\begin{cases}
        \displaystyle\frac{2 }{\log b} + \mathcal O\Bigl(\frac{1}{(\log b)^2}\Bigr) &\mbox{ if }\nu=0\,,\medskip\\
      \displaystyle  -|\nu| -  \frac{2 \ \Gamma(1-|\nu|) \ \Gamma(|\nu|+\half)}{\sqrt{\pi} \ \Gamma (|\nu|)} \  b^{|\nu|} +  \mathcal O(b^{2|\nu|})&\mbox{ if }\nu\neq0\,.
    \end{cases}\]
    \end{itemize}
\end{remark}
\subsection*{Organization} The paper is organized as follows:
\begin{itemize}
\item 
In Section~\ref{sec:proof-Thm1} we prove Theorem~\ref{thm:main1}, establishing strong-field asymptotics.
\item 
In Section~\ref{sec:applications} we derive applications of Theorem~\ref{thm:main1}, including Corollary~\ref{corol:main1}.
\item 
 In Section~\ref{sec:proof-Thm2}, we analyze the weak-field regime. We then complete the proof of Theorem~\ref{thm:gs-weakfield} using two different methods: one relying on the Temple inequality and the other employing special functions.
\end{itemize}
\section{Proof of Theorem~\ref{thm:main1}}\label{sec:proof-Thm1}
In this section, we derive the three-term asymptotic expansion for the magnetic Laplacian in the strong field limit. 
The proof proceeds in several steps. First, we recall basic properties of the de Gennes model on the half-line. Next, we perform a reduction to an annulus, followed by a translation and scaling that localize the problem near the boundary. Finally, a spectral reduction yields an effective operator whose eigenvalues are approximated by those of a harmonic oscillator with a flux-dependent shift.

\subsection{De Gennes model}~\\
For  $\gamma\in\R$. and $\xi\in\R$, we consider on $\R_+$ the operator
\[  \mathfrak{h}_0[\xi,\gamma]=-\frac{d^2}{dt^2}+(t-\xi)^2\]
 subject to the boundary condition $u'(0)=\gamma u(0)\,$. \\
  Let $\mu_0(\xi,\gamma)$ be its lowest eigenvalue, and let
\begin{equation}\label{eq:def-Th}
\Theta(\gamma)=\inf_{\xi\in\R}\mu_0(\xi,\gamma)\,.
\end{equation}
It is known that \cite[Theorem~II.2]{K-jmp}
\[\Theta(\gamma)=\mu_0(\xi,\gamma) \mbox{ if and only if }\xi=\xi(\gamma):=\sqrt{\Theta(\gamma)+\gamma^2}\,,\]
and that $\xi(\gamma)$ is a non-degenerate minimum of $\mu_0(\xi,\gamma)$.

From now on, we fix $\xi=\xi(\gamma)$ in the definition of $ \mathfrak h_0[\xi,\gamma]$ and introduce the operator
\begin{equation}\label{eq:def-h0}
 \mathfrak h_0(\gamma)=-\frac{d^2}{dt^2}+(t-\xi(\gamma))^2.
\end{equation}
Letting $\varphi_\gamma$ be the positive-valued normalized ground state of $\mathfrak h_0(\gamma)$, we have \cite[Proposition~II.5]{K-jmp}
\[\Theta'(\gamma)=|\varphi_\gamma(0)|^2. \]
Furthermore, we have for any $\gamma \in \R$  the following identities\footnote{There was an error in the calculation of the third moment in \cite[(2.21)]{K-jmp}, which is corrected in \cite[Lemma~B.3]{FTRV}.} (see \cite[(2.19)-(2.20)]{K-jmp} and \cite[Lemma~B.3]{FTRV}),
\begin{equation}\label{eq:moments}
\begin{gathered}
\int_{\R_+}(t-\xi(\gamma))|\varphi_\gamma(t)|^2\dd t=0\,,\\
\int_{\R_+}(t-\xi(\gamma))^2|\varphi_\gamma(t)|^2\dd t=\frac{\Theta(\gamma)}{2}-\frac{\gamma}{4}\,|\varphi_\gamma(0)|^2\,,\\
\int_{\R_+}(t-\xi(\gamma))^3|\varphi_\gamma(t)|^2\dd t=\frac16\bigl(1+2\gamma\xi(\gamma)\bigr)|\varphi_\gamma(0)|^2\,.
\end{gathered}
\end{equation}
We will need one more identity involving the regularized resolvent 
\begin{equation}\label{eq:defR0}
R_0(\gamma):=(\mathfrak h_0(\gamma)-\Theta(\gamma))^{-1}\,,
\end{equation} which is the inverse on the orthogonal complement of $\varphi_\gamma$  and is zero on $\mathbb R\, \varphi_\gamma$. 
\begin{proposition}
    \label{prop:identity} For any $\gamma \in \mathbb R$, we have
   \[\int_{\R_+}(t-\xi(\gamma))\varphi_\gamma\cdot R_0(\gamma)\bigl((t-\xi(\gamma))\varphi_\gamma\bigr)\dd t=-\frac14+\frac{\xi(\gamma)}{4}|\varphi_\gamma(0)|^2.\]
\end{proposition}
\begin{proof}
    Let $f=R_0(\gamma)\bigl((t-\xi(\gamma))\varphi_\gamma\bigr)$. Then, by definition of $R_0$,  $f$ is orthogonal to $\varphi_\gamma$, satisfies the boundary condition $f'(0)=\gamma f(0)$, and  the differential equation
    \[ (\mathfrak h_0(\gamma)-\Theta(\gamma))f=(t-\xi(\gamma))\varphi_\gamma\,.\]
Differentiating once the equation
\[(\mathfrak h_0(\gamma)-\Theta(\gamma))\varphi_\gamma=0\,,\]
we obtain
\[(\mathfrak h_0(\gamma)-\Theta(\gamma))\varphi'_\gamma=- 2 (t-\xi_\gamma) \varphi_\gamma \,.\]
Hence $f$ has the form
    \[
    f= -\frac 12\varphi_\gamma' + \rho \varphi_\gamma\,,
    \]
    where $\rho$ is chosen so that the right hand side is orthogonal to $\varphi_\gamma$ and satisfies the Robin condition.
    A priori, this could be strange because we get two conditions and one parameter,  but
    the Robin condition  is automatically satisfied (using the relations between $\xi(\gamma)$, $\Theta(\gamma)$ and $\gamma^2$).
    The orthogonality condition
     reads
     \[
     \rho=\frac 12 \int_0^{+\infty} \varphi_\gamma'(t)\, \varphi_\gamma(t)\, dt = -\frac 14 \varphi_\gamma(0)^2\,.
     \]
     Hence  we get for $f$:
    \[f=-\frac12\varphi_\gamma'-\frac14\, \varphi_\gamma(0)^2\varphi_\gamma\,.\]
    Using \eqref{eq:moments} and integrating by parts, we find that
    \[ \begin{aligned}
        \int_{\R_+}(t-\xi(\gamma))\varphi_\gamma(t)\cdot f(t)\,\dd t&=-\frac12\int_{\R_+}(t-\xi(\gamma))\varphi_\gamma(t)\cdot\varphi_\gamma'(t)\,\dd t\\
        &=-\frac14+\frac{\xi(\gamma)}{4}\varphi_\gamma(0)^2\,.
    \end{aligned}\]
\end{proof}
Finally, by Sturm-Liouville theory, the eigenvalues of $\mathfrak h_0[\xi,\gamma]$ are simple and
\begin{equation}\label{eq:gap}
\Theta_1(\gamma)=\inf_{\xi\in\R}\mu_1(\xi,\gamma)>\Theta(\gamma)\,.
\end{equation}
\subsection{Two-term asymptotics}~\\
The following two-term asymptotics
\begin{equation}\label{eq:2-term-exp}
\mu(b,\nu,\gamma)=\Theta(\gamma) b+ \mathcal C (\gamma) b^{1/2}+o(b^{1/2})\,,
\end{equation}
holds for $\nu=0$ (see \cite{K-jmp, FTRV}). It can be generalized to $\nu\in(-1/2,1/2]$ with the two first same terms by a slight adjustment of the argument in \cite{K-jmp}.
\par\noindent

\subsection{Reduction near the boundary}~\\
It is well known that the ground states of $\mathcal L$ (see \eqref{defHam}) 
decay exponentially away from the boundary (see \cite[Theorem~IV.1]{K-jmp} for $\nu=0$). Consequently,  modulo $\mathcal O(b^{-\infty})$,  the lowest eigenvalue of $\mathcal L$ in $\Omega$ is given by the lowest eigenvalue $\widetilde \mu(b,\nu,\gamma)$ in the annulus
\[\widetilde\Omega=\{x\in\R^2\colon 1<|x|<2\}\]
of the magnetic Laplacian (with same magnetic potential) submitted  to the Robin boundary condition on $\{|x|=1\}$ and to the Dirichlet boundary condition on $\{|x|=2\}$.

Moreover, for any given $ s\in(0,1)$, any normalized eigenfunction $\widetilde u$ in $\widetilde\Omega$ corresponding to an eigenvalue $\widetilde\mu\leq s  b$ decays away from the circle of radius $1$  as $b\rightarrow +\infty$. One way to quantify this decay is through the following estimate
\begin{equation}
    \label{eq:dec-ef}
    \forall\,n\in\N,\quad \int_{\widetilde\Omega}
    (r-1)^n|\widetilde u(x)|^2\dd x=\cO(b^{-n/2}) \mbox{ as } b \rightarrow +\infty\,.
\end{equation}
\subsection{Translation and scaling}
By separation of variables, we end up with the study of  the lowest eigenvalue $\widetilde\mu^{(m)}(b,\nu,\gamma)$ of the operator (indexed by $m\in\Z$)
\begin{subequations}
\begin{equation}
\widetilde{\mathcal H}^{(m)}=-\frac{d^2}{dr^2}-\frac1r\frac{d}{dr}+\Bigl(\frac{m-\nu}{r}-\frac{br}{2} \Bigr)^2 
\mbox{ on } L^2\bigl((1,2);r\dd r\bigr)\,,
\end{equation}
 with boundary conditions
 \begin{equation}
 u'(1)=b^{1/2}\gamma\,  u(1) \mbox{  and } u(2)=0\,.
 \end{equation}
 \end{subequations}
  In fact, the  lowest eigenvalue in the annulus $\widetilde\Omega$ is expressed as
\begin{equation}
 \widetilde\mu(b,\nu,\gamma)=\inf_{m\in\Z}\widetilde \mu^{(m)}(b,\nu,\gamma)\,,
 \end{equation}
and if $\widetilde u$ is a normalized ground state of $\widetilde{\mathcal H}^{(m)}$, then $\ee^{\ii m\theta}\widetilde u$ is an eigenfunction of $\mathcal L$ in $\widetilde\Omega$, with corresponding eigenvalue $\widetilde\mu^{(m)}(b,\nu,\gamma)$. Moreover, if $\widetilde\mu^{(m)}(b,\nu,\gamma)\leq sb$ with $0<s<1\,$,  then $\widetilde u$  satisfies the decay estimate stated in \eqref{eq:dec-ef}.

The change of variable $t=(r-1)b^{1/2}$ leads to  the operator
\begin{subequations}
\begin{equation}
\mathcal H^{(m)}=-\frac{d^2}{dt^2}-\frac{b^{-1/2}}{1+b^{-1/2}t}\,\frac{d}{dt}+\frac{b^{-1}}{(1+b^{-1/2}t)^2}\Bigl(m-\nu-\frac{b}{2}-b^{1/2}t-\frac{t^2}{2} \Bigr)^2\end{equation}
in $ 
L^2\bigl((0,b^{1/2});(1+b^{-1/2}t)\dd t\bigr)$, 
 subject to the boundary conditions
\begin{equation}
 u'(0)=\gamma \, u(0)\,,\quad u(b^{1/2})=0\,.
 \end{equation}
 \end{subequations}
The lowest eigenvalue $\mu^{(m)}(b,\nu,\gamma)$ of $\mathcal H^{(m)}$ is related to the lowest eigenvalue of $\widetilde{\mathcal H}^{(m)}$ by
\[\widetilde\mu^{(m)}(b,\nu,\gamma)=b\, \mu^{(m)}(b,\nu,\gamma)\,,\]
and if $\mu^{(m)}(b,\nu,\gamma)\leq s \,$,
then a normalized ground state $u$ of $\mathcal H^{(m)}$ satisfies
\begin{equation}\label{eq:dec-gs}
\forall\,n\in\N\,,\quad \int_{0}^{b^{1/2}}t^n\,|u(t)|^2\dd t=\mathcal O(1)\,.
\end{equation}
The quadratic form associated with  $\mathcal H^{(m)}$ is
\[\mathfrak q(u)=\int_0^{b^{1/2}}\Bigl(\Bigl| u'(t)\Bigr|^2+\frac{1}{(1+b^{-1/2}t)^2}V_{m,b}(t)|u(t)|^2\Bigr)(1+b^{-1/2}t)\dd t\,,\]
where
\[
    V_{m,b}(t)=b^{-1}\Bigl(m-\nu-\frac{b}{2}-b^{1/2}t-\frac{t^2}{2} \Bigr)^2\,.
\]
\subsection{Rough localization of angular momenta}~\\
In light of \eqref{eq:2-term-exp}, we focus on the $m\in\Z$ such that
\begin{equation}\label{eq:maj}
\mu^{(m)}(b,\nu,\gamma)\leq \Theta(\gamma)+\bigl(\mathcal C(\gamma)+1\bigr)b^{-1/2}.
\end{equation}
For $t\in(0,b^{1/2})$, we write
\[V_{m,b}(t)\geq  \frac{b^{-1}}{2}\Bigl(m-\nu-\frac{b}{2}\Bigr)^2-\mathcal O(t^2)-\mathcal O(b^{-1}t^4)\,,\]
and with $u$ a normalized ground state of $\mathcal H^{(m)}$, we use the decay estimate in \eqref{eq:dec-gs} to write
\[q_m(u)\geq \frac{b^{-1}}{4}\Bigl(m-\nu-\frac{b}{2}\Bigr)^2-\mathcal O(1)\,.\]
This yields a first localization of the minimizing  $m$,
\[
\Bigl|m-\nu-\frac{b}{2}\Bigr|\leq Mb^{1/2},
\]
where $M>0$ is a constant.

Our next aim is to refine this localization. With 
\begin{equation}\label{eq:defdelta}
\delta=\frac{m-\nu-\frac{b}{2}}{b^{1/2}}\in[-M,M]\,,
\end{equation}

we decompose the potential as
\[\begin{aligned}
    V_{m,b}(t)&=\Bigl(\delta-t-\frac{b^{-1/2}t^2}{2} \Bigr)^2\\
    &\geq (t-\delta)^2-Mb^{-1/2}t^2.
\end{aligned} \]
Using the decay of $u$ in \eqref{eq:dec-gs} and that $\xi(\gamma)$ is a non-degenerate minimum of $\mu_0(\xi,\gamma)$, we deduce that
\[\begin{aligned}
    q_m(u)&\geq \bigl(1+\mathcal O(b^{-1/2})\bigr)\mu_0(\delta,\gamma)-\mathcal O(b^{-1/2})\\
    &\geq \bigl(1+\mathcal O(b^{-1/2})\bigr)\bigl[\Theta(\gamma)+ \tilde c (\gamma) \bigl(\delta-\xi(\gamma)\bigr)^2\bigr]-\cO(b^{-1/2}),
\end{aligned}\]
where
\[\tilde c(\gamma)=\min_{\xi\in[-M,M]}\frac{\mu_0(\xi,\gamma)-\Theta(\gamma)}{(\xi-\xi(\gamma))^2}>0\,.\]
Consequently,   with \eqref{eq:maj} in mind and using that $ \tilde c (\gamma) >0$, we obtain first $$\bigl(\delta-\xi(\gamma)\bigr)^2=\mathcal O(b^{-1/2})\,,$$  and we then get the finer localization of $m$,
\begin{equation}\label{eq:loc-m}
\Bigl|m-\nu-\frac{b}{2}-b^{1/2}\xi(\gamma)\Bigr|\leq \tilde Mb^{1/4},
\end{equation}
where $\tilde M$ is a constant. For such integers $m$, a direct comparison argument with the harmonic oscillator $\mathfrak h_0[\delta,\gamma]$ 
yields that   the second eigenvalue of $\mathcal H^{(m)}$ satisfies
\[\mu_1^{(m)}(b,\nu,\gamma)\geq \Theta(\gamma)+\bigl(\mathcal C (\gamma)+2\bigr)b^{-1/2},\]
which essentially follows from \eqref{eq:gap} and \eqref{eq:dec-gs}.  In fact, if $u_1$ is a normalized eigenfunction corresponding to $\mu_1^{(m)}(b,\nu,\gamma)$, then it satisfies
\[ \forall\,n\in\N,\quad \int_0^{b^{1/2}}\bigl(|u_1'(t)|^2+|u_1(t)|^2 \bigr)t^n\,\dd t=\cO(1)\,.\]
Consequently, 
\[
    \mu_1^{(m)}(b,\nu,\gamma)=\mathfrak q(u_1)\geq \int_0^{b^{1/2}}\bigl(|u'_1(t)|^2+(t-\delta)^2|u_1(t)|^2\bigr)\dd t+\cO(b^{-1/2}\,),\]
    and we conclude by the min-max principle that
    \[\mu_1^{(m)}(b,\nu,\gamma)\geq \Theta_1(\gamma)+\cO(b^{-1/2})\,.\]
\subsection{Quasi-modes}~\\
We focus now on the case when angular momenta $m$ satisfying \eqref{eq:loc-m}, and decompose the operator $\mathcal H^{(m)}$ as
\begin{subequations}
\begin{equation}\label{eq:decomp-H} \mathcal H^{(m)}=\mathfrak h_0+b^{-1/2}\mathfrak h_1+b^{-1}\mathfrak h_2+\mathcal R\,,\end{equation}
\begin{equation}\label{eq:def-h012} 
\begin{aligned}
    \mathfrak h_0&=-\frac{d^2}{dt^2}+(t-\xi)^2\,,\\
    \mathfrak h_1&=-\frac{d}{dt}-2(t-\xi)\Bigl(\delta_2-\frac{t^2}{2}\Bigr)-2t(t-\xi)^2\,,\\
    \mathfrak h_2&=t\frac{d}{dt}+4t(t-\xi)\Bigl(\delta_2-\frac{t^2}{2}\Bigr)+
    3t^2(t-\xi)^2+\Bigl(\delta_2-\frac{t^2}{2}\Bigr)^2\,,
\end{aligned}
\end{equation}
where we wrote $$\mathfrak h_0=\mathfrak h_0(\gamma)\,, \quad \xi=\xi(\gamma)\,, \mbox{  and } \delta_2=m-\nu-\frac{b}2-b^{1/2}\xi\,,$$ to lighten the notation.\\
 The remainder $\mathcal R$ satisfies, for some constant $C$ independent of $b$ and $\delta_2$,
\begin{equation}
|(\mathcal R f)(t)|\leq C\, b^{-3/2}\Bigl(t^2\Bigl|f'(t)\Bigr|+(1+t^6+\delta_2^2)|f(t)|\Bigr).
\end{equation}
\end{subequations}
The aforementioned decomposition follows by expanding the potential
\[b^{-1}V_{m,b}(t)=(t-\xi)^2-2b^{-1/2}(t-\xi)\Bigl(\delta_2-\frac{t^2}{2}\Bigr)+b^{-1}\Bigl(\delta_2-\frac{t^2}{2}\Bigr)^2, \]
and the weights
\[ \begin{aligned}
    \frac{1}{1+b^{-1/2}t}&=1-b^{-1/2}t+\cO(b^{-1}t^2)\,,\\
    \frac{1}{(1+b^{-1/2}t)^2}&=1-2b^{-1/2}t+3b^{-1}t^2+\cO(b^{-3/2}t^3)\,.
\end{aligned}\]
Then we choose an approximate eigenpair $(v,\mu)$ such that
\[v=v_0+b^{-1/2}v_1+b^{-1}v_2,\quad \mu=\mu_0+b^{-1/2}\mu_1+b^{-1}\mu_2\,,\]
where $v_0,v_1,v_2$ belong to the domain of $\mathfrak h_0$ and to the Schwartz space $\mathcal S(\R_+)$. Hence they are required to satisfy, for $j=0,1,2$, the $\gamma$-Robin boundary condition   $v_j'(0)=\gamma\, v_j(0)\,$.\\
Solving formally 
\[ (\mathfrak h_0+b^{-1/2}\mathfrak h_1+b^{-1}\mathfrak h_2 )v=(\mu_0+b^{-1/2}\mu_1+b^{-1}\mu_2)v\,,\]
by equating the like powers of $b^{-1/2}$\,, we get
\[
\begin{gathered}
    (\mathfrak h_0 -\mu_0)v_0=0\,,\\
    (\mathfrak h_0-\mu_0)v_1=(\mu_1-\mathfrak h_1)v_0\,,\\
    (\mathfrak h_0-\mu_0)v_2=(\mu_2-\mathfrak h_2)v_0+(\mu_1-\mathfrak h_1)v_1\,.
\end{gathered}
\]
This leads to the following choices
\begin{equation}\label{eq:def-mu0}
\mu_0=\Theta(\gamma)\,,\quad v_0=\varphi_\gamma,\,
\end{equation}
and
\[
\begin{array}   {ll}
    \mu_1=\langle v_0,\mathfrak h_1v_0\rangle\,,\quad 
    v_1=-R_0(\gamma)(\mathfrak h_1v_0)\,,\\
    \mu_2=\langle v_0,\mathfrak h_2v_0\rangle+\langle v_0,(\mathfrak h_1-\mu_1)v_1\rangle\,,\\
    v_2=-R_0(\gamma)(\mathfrak h_2v_0)
    + R_0(\gamma)(\mu_1-\mathfrak h_1)v_1\,,
\end{array}
\]
where $R_0(\gamma)$ was introduced in \eqref{eq:defR0} and 
$\langle\cdot,\cdot\rangle$ is the inner product in $L^2(\R_+,\dd t)$. By a straightforward computation,
\[\mu_1=\frac{\varphi_\gamma(0)^2}{2}+\xi^2\int_0^{+\infty}(t-\xi)\varphi_\gamma(t)^2\dd t-\int_{0}^{+\infty}(t-\xi)^3\varphi_\gamma(t)^2\dd t\,,\]
and we get by \eqref{eq:moments},
\begin{equation}\label{eq:def-mu1}\mu_1=\frac13(1-\gamma\xi)\varphi_\gamma(0)^2=\mathcal C_{0}(\gamma)\,. 
\end{equation}
The calculation of $\mu_2$ is more subtle. The key is to note that it is a quadratic function of $\delta_2$, hence
\begin{equation}\label{eq:def-mu2}\mu_2=k_0+k_1\delta_2+k_2\delta_2^2\,.
\end{equation}
 Returning to the definitions of $\mathfrak h_1$ and $\mathfrak h_2$ in \eqref{eq:def-h012}, we note that $\mathfrak h_1$ and $\mathfrak h_2-\delta_2^2$ are  monomials in $\delta_2$.   Moreover, the  coefficient of $\delta_2$ in $\mathfrak h_1$ is $-2(t-\xi)\,$. Since $v_0,\mu_1$ are independent of $\delta_2$, and $v_1=-R_0(\gamma)\mathfrak (h_1v_0)$, we obtain that the coefficient of $\delta_2^2$ in $(\mathfrak h_1-\mu_1)v_1$ is
\[ -4(t-\xi) R_0(\gamma)(t-\xi)v_0\,,\]
while the coefficient of $\delta_2^2$ in $\mathfrak h_2v_0$ is $v_0\,$.

Consequently, we find that
\begin{equation}\label{eq:def-k2}
k_2=\langle v_0,v_0-4(t-\xi)R_0(\gamma)(t-\xi)v_0\rangle\,.
\end{equation}
Using the formula in Proposition~\ref{prop:identity}, we obtain
\[k_2=\xi(\gamma)\varphi_\gamma(0)^2>0\,.\]
Similarly, we can have explicit formulas for $k_1$ and $k_2$. In fact,
\[
\begin{aligned}
    k_1&=\bigl\langle v_0,(3t^2-4t\xi )v_0-2\bigl(\frac{d}{dt}+(t^2-2t\xi)(t-\xi)\bigr) R_0(\gamma)((t-\xi)v_0)\bigr\rangle,\\
    k_0&=\bigl\langle v_0,p_2v_0-p_1R_0(\gamma) (p_1 v_0)\bigr\rangle,
\end{aligned}
\]
where
\[ \begin{gathered}
    p_1=-\frac{d}{dt}+t^2(t-\xi)-2t(t-\xi)^2,\\
    p_2=t\frac{d}{dt}-2t^3(t-\xi)+3t^2(t-\xi)^2+\frac{t^4}{4}.
\end{gathered}\]
Now we can express $\mu_2$ as 
\[\mu_2=\xi(\gamma)\varphi_\gamma(0)^2\bigl[(\delta_2-\mathcal C_0(\gamma))^2+\mathcal C_1(\gamma)\bigr] \,,\]
where
\[\mathcal C_0(\gamma)=\frac{k_1}{2k_2}\,,\quad\mathcal C_1(\gamma)=\frac{k_0}{k_2}-\frac{k_1^2}{4k_2^2}\,,\]
are independent of $\nu$ and are defined explicitly in terms of the spectral parameters $\Theta(\gamma)$ and $\varphi_\gamma\,$.\\
Truncating the test function $v$, we get by the spectral theorem, for any $m$ satisfying \eqref{eq:loc-m},
\begin{equation}\label{eq:lb-m}
|\mu^{(m)}(b,\nu,\gamma)-(\mu_0+b^{-1/2}\mu_1+b^{-1}\mu_2)|\leq C\, (1+\delta_2^2)b^{-3/2}\,.
\end{equation}
From this we deduce that\footnote{Since $\delta_2=\cO(b^{1/4})$,  we cannot obtain a quantitative bound on the remainder at this stage.}
\begin{equation}\label{eq:lb-m*}
\mu^{(m)}(b,\nu,\gamma)=\mu_0+b^{-1/2}\mu_1+b^{-1}\mu_2+o(b^{-1})\,,
\end{equation}
uniformly with respect to the integers $m$ obeying \eqref{eq:loc-m}.

Minimizing over $m$ we get
\[\inf_{m\in\Z}\mu^{(m)}(b,\nu,\gamma)=\mu_0+b^{-1/2}\mu_1+b^{-1}\xi(\gamma)\varphi_\gamma(0)^2\inf_{m\in\Z}\Delta_m(b,\nu,\gamma)+o(b^{-1})\,,\]
where $\Delta_m(b,\nu,\gamma)$ is introduced in (\ref{eq:thm0}b).\\
Now we know that, in order to estimate $\inf_{m\in\Z}\mu^{(m)}(b,\nu,\gamma)$,  the relevant integers $m$ are those corresponding to $\delta_2=\cO(1)$, hence we deduce from \eqref{eq:lb-m} the expansion with a quantitative estimate of the remainder, namely
\[\inf_{m\in\Z}\mu^{(m)}(b,\nu,\gamma)=\mu_0+b^{-1/2}\mu_1+b^{-1}\xi(\gamma)\varphi_\gamma(0)^2\inf_{m\in\Z}\Delta_m(b,\nu,\gamma)+\cO(b^{-3/2})\,.\]
\section{Applications in the strong magnetic field limit}\label{sec:applications}
We now discuss a few applications of  Theorem~\ref{thm:main1} and its proof. Specifically, we present refined results for the spectral gap, the structure of eigenfunctions, and the lowest magnetic Steklov eigenvalue. In particular, we show how the flux-dependence in the third term of the Laplace eigenvalue translates into a flux-dependent correction for the Steklov eigenvalue through the notion of \(e_0\)-sequences.

\subsection{Spectral gap}
Let
\begin{equation}\label{eq:seq-ev}
\mu_0(b,\nu,\gamma)\leq \mu_1(b,\nu,\gamma)\leq \cdots
\end{equation}
denote the eigenvalues of the magnetic Laplacian $\cL$, repeated according to multiplicity. Theorem~\ref{thm:main1} gives the asymptotic behavior of the lowest eigenvalue $\mu(b,\nu,\gamma)=\mu_0(b,\nu,\gamma)$. This has followed  from the asymptotics in \eqref{eq:lb-m*}, which consequently enables the spectral analysis of the higher eigenvalues of $\cL$.

For $\gamma\in\R$ and $\nu\in(-1/2,1/2]\,$, we introduce
\begin{equation}\label{defeta}
 \eta(b,\nu,\gamma)=\frac{b}{2}+b^{1/2}\xi(\gamma)+\mathcal C_0(\gamma)+\nu\,,
 \end{equation}
 which is related to the quantity in \eqref{eq:defetabnu}. In fact,
\[\eta(b,\nu,\gamma)=\eta(b,\nu)\quad \mbox{for $\gamma=-\hat\alpha$\,.}\]
We can then generalize the notion of $e_0$-sequences: 
\begin{definition}
Let $e_0\in(-1/2,1/2]\,.$ A sequence $(b_n)$ is said to be an $e_0$-sequence corresponding to $(\nu,\gamma)$ if  $b_n \rightarrow +\infty$ as $n\rightarrow +\infty$ and if there is a sequence $(p_n)$ of integers  such that 
\[\eta(b_n,\nu,\gamma)-p_n\to e_0\,. \]
\end{definition}
\begin{theorem}\label{thm:main1*}
Let $\gamma\in\R$,  $\nu\in(-1/2,1/2]$ and  $e_0\in(-1/2,1/2]\,$. Let $(b_n)$ be an $e_0$-sequence corresponding to $(\gamma,\nu)$.
%
Then, it holds the following:
\begin{enumerate}[\rm 1.]
    \item The lowest eigenvalue of $\cL$ satisfies
    \[\mu_0=\Theta(\gamma)b_n+\mathcal C(\gamma)b^{1/2}_n+\bigl(e_0^2+\mathcal C_1(\gamma)\bigr)\xi(\gamma)\Theta'(\gamma)
     +\mathcal O(b^{-1/2}_n)\,.\]
     \item The second eigenvalue of $\cL$ satisfies
     \[\mu_1=\Theta(\gamma)b_n+\mathcal C(\gamma)b^{1/2}_n+\bigl((1-|e_0|)^2+\mathcal C_1(\gamma)\bigr)\xi(\gamma)\Theta'(\gamma)
     +\mathcal O(b^{-1/2}_n)\,.\]
     \item The third eigenvalue satisfies
     \[\mu_2=\Theta(\gamma)b_n+\mathcal C(\gamma)b^{1/2}_n+\bigl((1+|e_0|)^2+\mathcal C_1(\gamma)\bigr)\xi(\gamma)\Theta'(\gamma)
     +\mathcal O(b^{-1/2}_n)\,.\]
\end{enumerate}
\end{theorem}
In the Neumann case ($\gamma=0$) with no additional flux term $\nu=0$, we recover \cite[Theorem~2]{FM}. 

The sequence $(b_n)$ in Theorem~\ref{thm:main1*} ensures that  
\[
\inf_{m\in\Z}|m-\eta(b_n,\nu,\gamma)| \to |e_0|\,.
\]  
Consequently, the oscillatory term in Theorem~\ref{thm:main1} satisfies  
\[
\Delta_m(b_n,\nu,\gamma) \to |e_0|^2 + \mathcal{C}_1(\gamma)\,.
\]  
Moreover, to leading order, the quantities $|e_0|$, $1-|e_0|$, and $1+|e_0|$ represent respectively the distances from $\eta(b_n,\nu,\gamma)$ to the closest, second closest, and third closest integers. This explains why the asymptotics of the second and third eigenvalues follow from \eqref{eq:lb-m*}.

\subsection{Structure of eigenfunctions}%

An orthonormal basis of eigenfunctions of the magnetic Laplacian is given by
\begin{equation}\label{eq:def-psi-m}
\psi_{m,n}=\ee^{\ii m\theta}f_n(r)\,,
\end{equation}
where $m\in\Z$ and $(f_n)$ is an orthonormal basis  of eigenfunctions of the operator 
\begin{equation}\label{eq:def-Hm}
{ \mathcal L}^{(m)}=-\frac{d^2}{dr^2} -\frac1r\frac{d}{dr}+\Bigl(\frac{m-\nu}{r}-\frac{br}{2} \Bigr)^2 
\mbox{ on } L^2\bigl((1,+\infty);r\dd r\bigr)\,,
\end{equation}
subject to the Robin boundary condition $u'(1)=b^{1/2}\gamma u(1)$. The ground states correspond to the $m$ that minimizes the ground state energy of $\mathcal L^{(m)}$. In the limit of large $b$, the minimizing $m$ satisfies
\[m\in\{m_{-}(b,\nu,\gamma),m_+(b,\nu,\gamma)\},\]
where
\[m_-(b,\nu)=\lfloor \eta(b,\nu,\gamma)\rfloor,\quad m_+(b,\nu,\gamma)=m_-(b,\nu,\gamma)+1\,.\]
This follows from  \eqref{eq:lb-m*} and (\ref{eq:thm0}b). For a given $e_0$-sequence $(b_n)$, the ground state energy is simple if $|e_0|<1/2$, while it can be multiple if $e_0=1/2$.

For $m=m_\pm(b,\nu,\gamma)$, we denote by $f_\pm(r)$ the corresponding  ground states of $\mathcal H^{(m)}$, and by $\psi_\pm$ the corresponding functions in \eqref{eq:def-psi-m}.

In the Neumann case ($\gamma=0$), we have by \cite[Lemma~7]{FM}.
\begin{proposition}\label{prop:ground-states}
Suppose that $\gamma=0$ and $\nu\in(-1/2,1/2]$. Let  $m\in\{m_-(b,\nu,0),m_+(b,\nu,0)\}$. The normalized  eigenfunction
\[\psi_\pm=\ee^{\ii m_\pm(b,\nu,0)\theta}f_\pm(r)\]
satisfies
\[ f_\pm(r)=b^{\Theta_0/4}r^{m_\pm(b,\nu,0)}\ee^{-b/4(r^2-1)}u_\pm(r),\]
where
\[u_\pm(r)=K_0\,\Gamma(\delta_0)\, \left(\frac{2}{r^2-1}\right)^{\delta_0}\left(1+o\bigl(b^{-1/2}(r^2-1)^{-1}\bigr)\right),\]
locally uniformly on $(1,+\infty)$.  Here  $\delta_0=\frac{1-\Theta_0}{2}$,  $K_0$ is a universal constant, and $\Gamma$ is the Gamma function.
\end{proposition}
\subsection{The Steklov eigenvalue}
We introduce
\begin{equation}
    \label{eq:def-gam(b,nu)}
    \gamma(b,\nu)=-b^{-1/2}\lambda(b,\nu)\,.
\end{equation}
With $\gamma=\gamma(b,\nu)$, we have $\mu(b,\nu,\gamma)=0$. With the help of Theorem~\ref{thm:main1} and \eqref{eq:asym-HN}, we prove  an expansion of the Steklov eigenvalue with three terms. 

\begin{theorem}\label{thm:main2}
Suppose that $\nu\in(-1/2,1/2]$ is fixed. Then, as $b\to+\infty$, the lowest magnetic eigenvalue satisfies,
\[\lambda(b,\nu)=\hat\alpha b^{1/2}+\frac{\hat\alpha^2+1}{3}+F(b,\nu)b^{-1/2}+\mathcal O(b^{-1})\,,\]
where
\[F(b,\nu):=\hat\alpha
\inf_{m\in\Z}\Delta_m\bigl(b,\nu,-b^{-1/2}\lambda(b,\nu)\bigr)+\mathcal K_0\,,\]
for some constant $\mathcal K_0$.
\end{theorem}
\begin{proof}
We write $\gamma=\gamma(b,\nu)$ and $\lambda=\lambda(b,\nu)$. Thanks to \eqref{eq:asym-HN}, we have
\[ \gamma=-\hat\alpha-\frac{\hat\alpha^2+1}{3}b^{-1/2}+\cO(b^{-1})\,.\]
 Let \begin{subequations}\label{eq:defdelta2}
\begin{equation}  \delta:=b^{1/2}(\gamma+\hat\alpha)\,,
\end{equation}
which consequently satisfies
\begin{equation}
\delta = -\frac{\hat\alpha^2+1}{3}+ \mathcal O (b^{-1/2})\,.
\end{equation}
\end{subequations}

 Knowing  (see \eqref{eq:xi-C-alpha})  that 
\[\Theta(-\hat\alpha)=0,\quad\xi(-\hat\alpha)=\hat\alpha,\quad \mathcal C(-\hat\alpha)=\frac13(\hat\alpha^2+1)\Theta'(-\hat\alpha)\,,\] 
we have
\[
\begin{gathered}
\Theta(\gamma)b=\delta\Theta'(-\hat\alpha)b^{1/2}+\frac12\delta^2\Theta''(-\hat\alpha)+\cO(b^{-1/2})\,,\\
\mathcal C(\gamma)b^{1/2}=\frac{\hat\alpha^2+1}{3}\Theta'(-\hat\alpha)b^{1/2}+\delta \mathcal C'(-\hat\alpha)+ \cO(b^{-1/2})\,,\\
\xi(\gamma)\Theta'(\gamma)=\hat\alpha\,\Theta'(-\hat\alpha)+\cO(b^{-1/2})\,,\\
\xi(\gamma)\Theta'(\gamma)\inf_{m\in\Z}\Delta_m(b,\nu,\gamma)=
\hat\alpha \,F(b,\nu)\, \Theta'(-\hat\alpha)+\cO(b^{-1/2})=\mathcal O(1)\,.
\end{gathered}
\]
Inserting these into the asymptotics in Theorem~\ref{thm:main1} and using the equation $\mu(b,\nu,\gamma)=0$\,, we get
\[ \frac{1}{2}\frac{\Theta''(-\hat\alpha)}{\Theta'(-\hat\alpha)}\delta^2+\left(b^{1/2}+\frac{\mathcal C'(-\hat\alpha)}{\Theta'(-\hat\alpha)}\right)\delta+\frac{\hat\alpha^2+1}{3}b^{1/2}+\hat\alpha F(b,\nu)=\cO(b^{-1/2})\,.\]
With $\mathcal M:=\Theta'(-\hat\alpha)/\Theta''(-\hat\alpha)$, 
we obtain by completing the square,
\begin{multline*}
    \left(\delta+\Bigl(b^{1/2}+\frac{\mathcal C'(-\hat\alpha)}{\Theta'(-\hat\alpha)}\Bigr)\mathcal M\right)^2=\\\Bigl(b^{1/2}+\frac{\mathcal C'(-\hat\alpha)}{\Theta'(-\hat\alpha)}\Bigr)^2\mathcal M^2-2\mathcal M\frac{\hat\alpha^2+1}{3}b^{1/2}-2\mathcal M\hat\alpha F(b,\nu)+\cO(b^{-1/2})\,,
\end{multline*}
which eventually yields
\[\delta=-\frac{\hat\alpha^2+1}{3}-\bigl(\mathcal K_0+\hat\alpha F(b,\nu)\bigr) b^{-1/2}+\cO(b^{-1})\,,\]
where $\mathcal K_0$ is defined by
\begin{equation}\label{eq:def-K0}
\mathcal K_0=\Bigl(\frac{\hat\alpha^2+1}{6\mathcal M}-\frac{\mathcal C'(-\hat\alpha)}{\Theta'(-\hat\alpha)}\Bigr)\frac{\hat\alpha^2+1}{3}.
\end{equation}
To finish the proof, we recall that $ \delta=b^{1/2}(\gamma+\hat\alpha)$ and that $\gamma$ is given by \eqref{eq:def-gam(b,nu)}.
\end{proof}
To prove the corollary, notice  that evidently $F(b,\nu)=\mathcal O(1)\,$.
We now observe   that \eqref{eq:asym-HN} yields
\[\begin{gathered}\xi(\gamma(b,\nu))=
\xi(-\hat\alpha)-\frac{\hat\alpha^2+1}{3}\xi'(-\hat\alpha)b^{-1/2}+\mathcal O(b^{-1})\,,\\
\mathcal C_i(\gamma(b,\nu))=\mathcal C_i(-\hat\alpha)+\cO(b^{-1/2})\mbox{ for } i=0,1\;,
\end{gathered}\]
and we obtain Corollary~\ref{corol:main1}  immediately from Theorem~\ref{thm:main2}. We have also used that 
\[\xi'(-\hat\alpha)=\frac{\Theta'(-\hat\alpha)-2\hat\alpha}{2\hat\alpha}\,.\]
which results from differentiating the identity  $\xi(\gamma)=\sqrt{\Theta(\gamma)+\gamma^2}$ and from $\Theta(-\hat\alpha)=0$.

\begin{remark}\label{rem:HN}
   It was established  in \cite[Theorem~1.4]{HN}  that $b\to\lambda(b,\nu)$ is increasing on $\R_+$.  
\end{remark}

\section{Weak magnetic  field limit for the exterior of the disk}\label{sec:proof-Thm2}

In this section, we investigate the low-lying eigenvalues introduced in \eqref{eq:seq-ev} in the Neumann case ($\gamma=0$) and in the weak magnetic field limit \(b \to 0^+\). The main result is stated as Theorem~\ref{thm:weakfield}.

The analysis is based on the study of the dispersion curves \(\mu_0^{(m)}(b,\nu)\) of the fiber operator \(\mathcal{L}^{(m)}\). After establishing ordering properties for these curves in Proposition~\ref{prop:weakfield}, we construct an effective Schr\"odinger operator \(S_\nu^{(m)}\) that captures the leading-order behavior. A quasi-mode argument combined with Temple's inequality yields the precise asymptotics of Theorem~\ref{thm:weakfield}. An alternative approach using confluent hypergeometric functions is also presented in Subsection~\ref{sec:special functions}.

\subsection{Main statement}~\\
We establish  accurate asymptotics  that display the eigenvalue splitting in the limit $b\to0^+$.

\begin{theorem}\label{thm:weakfield}
Let $\nu\in(-1/2,1/2]$. There exists $b_0>0 $ such that, for $b\in(0,b_0)$, we have the following.
\begin{enumerate}[\rm 1.]
\item The ground state energy $\mu_0(b,\nu,0)$ is a simple eigenvalue.%
\item The ground states are radially symmetric when $\nu<0$, and are not radially symmetric when $\nu\geq 0$.
\item If $\nu<0$, then for any fixed non-negative integer $k$, we have as $b\to0^+$
    \[
    \mu_k(b,\nu,0)=
    \begin{cases}
    \displaystyle b-\frac{2^{1+\nu}}{\Gamma(-\nu)}b^{1-\nu}+  \cO(b^{1-2\nu})&\mbox{if $k=0$\,,}\\
    \displaystyle b-\frac{1}{2^{k-\nu}\Gamma(k-\nu+1)}b^{k-\nu+2}+\cO(b^{k-\nu+\frac{5}2})&\mbox{if $k\geq 1\,$.}
    \end{cases}
    \]
    \item If $\nu\geq 0$, then for any fixed non-negative integer $k$, we have 
    \[\mu_k(b,\nu,0)=b-\frac{1}{2^{k-\nu}\Gamma(k-\nu+1)}b^{k-\nu+2}+\cO(b^{k-\nu+\frac{5}2}) \mbox{ as } b\to0^+\,.\]
\end{enumerate} 
\end{theorem}
\begin{remark}~
    \begin{enumerate}[\rm i.]
        \item In the case $\nu=0$, we recover Theorem~1.1 in \cite{KLS}.
        \item The ground state energy asymptotics  in Theorem~\ref{thm:weakfield} match for $\nu=\pm 1/2$. Indeed
        \[\mu_0(b,\pm1/2,0)=b-\sqrt{\frac{2}{\pi}}b^{3/2}+{ \cO(b^{2})}.\]
        \item Theorem~\ref{thm:weakfield} shows a lack of continuity at $\nu=0$.  As a function of $\nu$, the normalized ground state is left discontinuous at $\nu=0$ due to the change in radial symmetry. Moreover, while the function
        \[(b,\nu)\mapsto \frac{\mu_0(b,\nu,0)}{b}-1\]
        is continuous on $\R_+\times(-1/2,1/2]$, it cannot be extended by continuity to $\overline{\R_+}\times(-1/2,1/2]$, due to the discontinuity at $\nu=0\,$.
        \item In the case of the disk, the eigenvalues converge to those of the Laplace operator with Aharonov-Bohm potential (with flux $\nu$). If furthermore $\nu=0$, an accurate asymptotics  for the ground state energy is established in \cite[Proposition~1.5.2]{FH-b} (the leading order term is of order $b^2$).
    \end{enumerate}
\end{remark}
\subsection{Analysis of dispersion curves}
As in \cite{KLS}, the proof of Theorem~\ref{thm:weakfield} relies on analyzing the lowest eigenvalues $\mu^{(m)}_0(b,\nu)$ of the fiber operator $\cL^{(m)}$ in \eqref{eq:def-Hm} (called dispersion curves). Note that we impose Neumann boundary condition ($u'(1)=0$), and by Sturm-Liouville theory, $\mu^{(m)}_0(b,\nu)$ is a simple eigenvalue with a positive ground state.
\begin{proposition}
    \label{prop:weakfield}
    Let $\nu\in(-1/2,1/2]$. For $b>0$, the following holds.
    \begin{enumerate}[1.]
        \item The lowest eigenvalue $\mu_0^{(m)}(b,\nu)$ of $\cL^{(m)}$ 
        satisfies
        \[\begin{cases}
            \mu_0^{(m)}(b,\nu)>b&\mbox{ if }b>2(m-\nu)\,,\\
            \mu_0^{(m)}(b,\nu)=b&\mbox{ if }b=2(m-\nu)\,,\\
            \mu_0^{(m)}(b,\nu)<b&\mbox{ if }b<2(m-\nu)\,.\\
        \end{cases}\]
        \item If $m\geq 1$, the second eigenvalue of $\mathcal L^{(m)}$ satisfies $\mu_1^{(m)}(b,\nu)>b\,$.
    \end{enumerate}
\end{proposition}
\begin{proof}
The proof is essentially the same as \cite[Proposition~2.1]{KLS} devoted to the case $\nu=0$. To handle the case $\nu\not=0$, we
replace $m$ by $m-\nu$ in \cite[Proposition~2.1]{KLS}. We provide the details for the convenience of the reader.

Recall that $\cL^{(m)}=-\frac1r\partial_r(r\partial_r)+V_{m,b}(r)\,,$ where
\[ V_{m.b}(r)=\left(\frac{m-\nu}{r}-\frac{br}{2}\right)^2.\]
1. For $m-\nu<0$, we have
\[V_{m,b}(r)=\left(\frac{|m-\nu|}{r}-\frac{br}{2}\right)^2+2|m-\nu|b\,,\]
and by the min-max principle, we get
\[\mu_0^{(m)}(b,\nu)> 2|m-\nu|b\,.\]
Consequently, $\mu_0^{(m)}(b,\nu)>b$ holds for
$m\leq -1$ and $\nu\in(-1/2,1/2]$.\\
2. The general solution of the differential equation $\cL^{(m)}u=bu$ on $(1,+\infty)$ is
\[u(r)=r^{m-\nu}\ee^{-br^2/4}\left[c_1+c_2\int_1^r\rho^{-1-2(m-\nu)}\ee^{b\rho^2/2}\right], \]
where $c_1,c_2\in\R$ are constants. The solution is in $L^2((1,+\infty),r\dd r)$ if and only if $c_2=0$, and it satisfies the Neumann condition $u'(1)=0$ if and only if $b=2(m-\nu)$. Consequently, we have
\begin{itemize}\item $b\in\sigma(\cL^{(m)})$ if and only if $m-\nu>0$ and $b=2(m-\nu)\,$;
\item{ if $m-\nu>0$ and $b=2(m-\nu)$, then $\mu_0^{(m)}(b,\nu)=b\,$.}
\end{itemize}
Note that the second item  says more than $2(m-\nu)\in \sigma(\cL^{(m)})\,$. This is true by continuity because we know that $\mu^{(m)}_0(b,\nu)>b$ for $b$ sufficiently large.\\
3. In the case where $m=0$ and $\nu\in(0,1/2]$, we have $\mu_0^{(0)}(b,\nu)\not=b$ for all $b>0$. Since
\[V_{0}(r)\geq \left(\frac{b}{4}+\nu\right)b\quad\mbox{ for $r>1$\,,} \]
we get by the min-max principle that $\mu_0^{(0)}(b,\nu)>b$ for $b>4(1-\nu)\,$, and by continuity, $\mu^{(0)}(b,\nu)>0$  for all $b>0$\,.\\
4. Suppose that $m-\nu>0\,$. Let $g(b)=\mu_0^{(m)}(b,\nu)-b$ and let $u_m$ be the positive normalized ground state of $\cL^{(m)}$. Perturbation theory and the Feynmann-Hellman formula yield
\[g'(b)= \frac{\dd}{\dd b}\mu_0^{(m)}(b,\nu)-1=-\int_{1}^{+\infty}\left(m-\nu-\frac{br^2}{2}+1\right)|u_m(r)|^2r\dd r\,.\]
For $b=2(m-\nu)$, $u_m(r)=c_1r^{m-\nu}\ee^{-br^2/4}$, we have 
\[g'(2m-2\nu)=c_1^2\int_1^{+\infty}\bigl((m-\nu)(r^2-1)-1 \bigr)r^{1+2(m-\nu)}\ee^{-(m-\nu)r^2}\dd r\,.\] 
Note that  for $k>0$,
\[\begin{aligned}
\int_1^{+\infty}k(r^2-1)r^{1+2k}\ee^{-kr^2}\dd r
&=-\frac12\int_1^{+\infty}\left(r^{2k+2}-r^{2k}\right)(\ee^{-kr^2})'\dd r\\
&=\int_1^{+\infty}\left(k+1-\frac{k}{r^{2}}\right)r^{2k+1}\ee^{-kr^2}\dd r\\
&>\int_1^{+\infty}r^{2k+1}\ee^{-kr^2}\dd r\,.\end{aligned}\]
Consequently, for $m-\nu>0\,$, we have
\[g'(2m-2\nu)>0\,.\]
Since $b=2(m-\nu)$ is the unique zero of $g$, we conclude that $g(b)>0$ for $b<2(m-\nu)$ and $g(b)<0$ for $b>2(m-\nu)\,$.\\
4. Suppose that $m\geq 1$. By the Sturm-Liouville theory, we have the inequality $\mu_1^{(0)}(b,\nu)>\mu_0^{(m)}(b,\nu)\,$. For $b\geq 2(m-\nu)\,$, we know that $\mu_0^{(m)}(b,\nu)\geq b\,$, and for $b<2(m-\nu)\,$, we know that $\mu_1^{(m)}(b,\nu)\not=b\,$. By continuity, we should have $\mu_1^{(m)}(b,\nu)>b$ for all $b>0\,$. 
\end{proof}
In light of the relation
\[\mu_0(b,\nu,0)=\inf_{m\in\Z}\mu_0^{(m)}(b,\nu,0)\,,\]
we have the following immediate consequence
of Proposition~\ref{prop:weakfield}.
\begin{corollary}\label{corol:weakfield}
    If $0<b<|\nu|$\,, then the lowest eigenvalue  of the magnetic Neumann Laplacian satisfies: 
        \[ \mu(b,\nu,0)=\inf_{m\geq 0}\mu_0^{(m)}(b,\nu).\]
        Moreover, for $\nu\geq 0\,$, 
        \[ \mu(b,\nu,0)=\inf_{m\geq 1}\mu_0^{(m)}(b,\nu)\,,\]
        and the corresponding ground states are not radially symmetric.
\end{corollary}

Another useful property concerns the ordering of the eigenvalues.
\begin{proposition}\label{prop:ordering}
Let $\nu\in(-1/2,1/2]$ and suppose that $m-\nu> 2\,$. If 
\[0<b<2(m-\nu)+1-\sqrt{8(m-\nu)+1}\,,\] then 
\[\mu_0^{(m-1)}(b,\nu)<\mu_0^{(m)}(b,\nu)\,.\]
\end{proposition}
\begin{proof}
   The proof is the same as \cite[Proposition~2.10]{KLS}  with $m$ replaced by $m-\nu$.
\end{proof}

\subsection{An effective operator}\label{sub-sec:eo}
Let $m$ be a non-negative integer and $\nu\in(-1/2,1/2]\setminus\{0\}$. We introduce an effective operator $\mathcal S_\nu^{(m)}$ that arises when we zoom in on the boundary \(r=1\).  After a suitable unitary transformation and rescaling, we obtain the operator \(S_\nu^{(m)}\) from 
\(b^{-1}\mathcal{L}^{(m)}\) by taking the limit  $b\to0^+$. The influence of the additional flux  $(\nu\not=0)$ appears when computing  the lowest eigenvalue of $\mathcal S_\nu^{(m)}$. 

The operator $\mathcal S_\nu^{(m)}$ is the following Schr\"odinger operator on $\R_+$,
\begin{equation}\label{eq:def-op-S}
\mathcal S_\nu^{(m)}=-\frac{\dd^2}{\dd r^2}+w_\nu^{(m)}
\end{equation}
with the singular potential
\[ w_\nu^{(m)}(r)=\frac{4(m-\nu)^2-1}{4r^2}+r^2-2(m-\nu)\,,\]
and subject to the Dirichlet boundary condition $u(0)=0\,$.\\
 More precisely, $S_\nu^{(m)}$ is the self-adjoint
operator associated with the closed, densely defined, non-negative quadratic
form\footnote{Thanks to the Hardy inequality $\int_{\R_+}r^{-2}|f(r)|^2\dd r\leq 4\int_{\R_+}|f'(r)|^2\dd r$ for $ f\in H^1_0(\R_+)\,$, we get $r^{-1}f\in L^2(\R_+)$ and $q_\nu^{(m)}(f)\geq 0\,$.} 
\begin{equation}\label{eq:def-qf-S}
\begin{gathered}
    q_\nu^{(m)}(f)=\int_{\R_+}\bigl(|f'(r)|^2+w_\nu^{(m)}|f(r)|^2\bigr)\dd r\,,\\
    \mathrm{Dom}(q^{(m)}_\nu)=\{f\in H^1_0(\R_+)\colon rf\in L^2(\R_+)\}.
\end{gathered}
\end{equation}
Consequently, the domain of $\mathcal S_\nu^{(m)}$ is
\[\mathrm{Dom}(\mathcal S_\nu^{(m)})=\{f\in H^1_0(\R_+)\colon rf,\mathcal S_\nu^{(m)}f\in L^2(\R_+)\}.\]
\begin{proposition}\label{prop:sp-S}
For $\nu\in(-1/2,1/2]\setminus\{0\}$ and $m\geq 0$, the spectrum of $\mathcal S_\nu^{(m)}$ is purely discrete. Moreover, we have:
\begin{enumerate}[\rm 1.]
\item The lowest eigenvalue of $\mathcal S_\nu^{(0)}$ is
\[\lambda_0(\mathcal S_\nu^{(0)})=
\begin{cases}
4\nu+2&\mbox{if }\nu> 0\,,\\
4\nu+6&\mbox{if }\nu<0\,.
\end{cases}
\]
    \item  If $m\geq 1$, the lowest and second eigenvalues of $\mathcal S_\nu^{(m)}$ are respectively 
    \[\lambda_0(\mathcal S_\nu^{(m)})=2\,,\quad \lambda_1(\mathcal S_\nu^{(m)})=6\,.\]
\end{enumerate}
\end{proposition}
\begin{proof}
    The operator $\mathcal S_\nu^{(m)}$ is unitarily equivalent to the following self-adjoint operator in $L^2(\R_+,r\dd r)$,
\[H^{(m)}=-\frac{\dd^2}{\dd r^2}-\frac1r\frac{\dd}{\dd r}+\left(\frac{m-\nu}{r}-r\right)^2\,, \]
with domain
\[\mathrm{Dom}(H^{(m)})=\{u\colon u,H^{(m)}u\in L^2(\R_+,r\dd r),~u(0)=0\}\,.\]
In fact, $f\in\mathrm{Dom}(\mathcal S_\nu^{(m)})$ if and only if $u=r^{-1/2}f\in\mathrm{Dom}(H^{(m)})$, and we have the identity
\[r^{-1/2}\mathcal S_\nu^{(m)}r^{1/2}=H^{(m)}.\]
The operator $H^{(m)}$ has compact resolvent and by \cite[Sec.~III]{ESV}, its spectrum consists of the eigenvalues 
\[\lambda_{m,n}(\nu)=\begin{cases}
    2(\nu-m+|\nu-m|+2n+1)&\mbox{ if }\nu\geq 0\,,\\
    2(\nu-m+|\nu-m+1|+2n+2)&\mbox{ if }\nu<0\,,
\end{cases}\]
where $n\in\N_0$ ($\N_0$ is the set of integers $\geq 0$).
\end{proof}
\begin{remark}\label{rem:sp-S}
Another insight on the operator $H^{(m)}$ is obtained from the identity
\[
H^{(m)}f(r)=\ee^{-\ii m\theta}(-i\nabla-\Ab)^2\ee^{\ii m\theta}f(r),
\]
where 
\[\Ab(x)=(-x_2,x_1)+\frac{\nu}{|x|^2}(-x_2,x_1)\,.\]
When $\nu=0$, we recover the Landau Hamiltonian with magnetic field $2$ and the spectrum consists of the Landau levels $2,6,\cdots$, whereas 
when $\nu\not=0$, we get the Landau Hamiltonian with Aharonov-Bohm solenoid studied in \cite{ESV}.
\end{remark}

\subsection{Leading order asymptotics}\label{sub-sec:la}

\begin{proposition}
    \label{prop:weakfield-m}
    For $\nu\in(-1/2,1/2]\setminus\{0\}$, we have:
    \begin{enumerate}[\rm 1.]
        \item If $m\geq 1$, then the first  and second eigenvalues of the operator $\mathcal L^{(m)}$ satisfy
        \[\mu_0^{(m)}(b,\nu)=b+o(b),\quad \mu_1^{(m)}(b,\nu)=3b+o(b)\mbox{ as } b\to0^+\,.\]
        \item If $\nu\in (-1/2,0)$, then the first and second eigenvalues of  the operator $\mathcal L^{(0)}$ satisfy
        \[\mu_1^{(0)}(b,\nu)-\mu_0^{(0)}(b,\nu)\geq (2+\nu)b+o(b) \mbox{ as } b\to0^+\,.\]
    \end{enumerate}
\end{proposition}
\begin{remark}\label{rem:weakfield-m}  Proposition~\ref{prop:weakfield-m} excludes $\mathcal{L}^{(0)}$ for $\nu \geq 0$ because its spectrum does not contribute to the low-lying eigenvalues of the magnetic Neumann Laplacian. Indeed, Proposition~\ref{prop:weakfield} shows that $\mu_0^{(0)}(b,\nu) > b$ when $\nu \geq 0$ (see also Corollary~\ref{corol:weakfield}).
\end{remark}
\begin{proof}[Proof of Proposition~\ref{prop:weakfield-m}]~\\
    1. (\cite[Section~2.2.1]{KLS}) The unitary transformation $L^2((1,+\infty),\dd r)\ni f\mapsto r^{-1/2}f\in L^2((1,+\infty),r\dd r)$ and the change of variable $r\mapsto \sqrt{\frac{b}{2}}(r-1)$ yield that $\mathcal L^{(m)}$ is unitarily equivalent to the operator $(b/2)\mathcal S_{b,\nu}^{(m)}$, where $\mathcal S_{b,\nu}^{(m)}$ is the Schr\"odinger operator in $L^2(\R_+)$ defined by
    \[\begin{gathered}
        \mathcal S_{b,\nu}^{(m)}=-\frac{\dd^2}{\dd r^2}+w_{b,\nu}^{(m)}\,,\\
        w_{b,\nu}^{(m)}(r)=\frac{4(m-\nu)^2-1}{4\left(r+\sqrt{\frac{b}{2}}\right)^2}+\left(r+\sqrt{\frac{b}{2}}\right)^2-2(m-\nu)\,,
    \end{gathered}\]
    with domain
    \[ \mathrm{Dom}(\mathcal S_{b,\nu}^{(m)})=\{f\colon f,\mathcal S_{b,\nu}^{(m)}f\in L^2(\R_+),~f'(0)=\sqrt{\frac{1}{2b}}f(0)\}\,.\]
    2. For $m\geq 1$, the singular term in the potential $w_{b,\nu}^{(m)}(r)$ is non-negative  and thus $w_{b,\nu}^{(m)}(r)$ is bounded from below by $r^2-2(m-\nu)$. Consequently, we can repeat the proof of \cite[Proposition~2.5]{KLS}   with $m$ replaced by $m-\nu$, and we obtain that as $b\to0^+$ the operator $\mathcal S_{b,\nu}^{(m)}$ converges in the strong resolvent sense  to $\mathcal S_{\nu}^{(m)}$ introduced in \eqref{eq:def-op-S}, and that the eigenvalues of $\mathcal S_{b,\nu}^{(m)}$ converge to the corresponding eigenvalues of $\mathcal S_{\nu}^{(m)}$. Thanks to Proposition~\ref{prop:sp-S}, we obtain the first assertion in Proposition~\ref{prop:weakfield-m}.\\
    3. For $m=0$, we cannot repeat the proof of \cite[Proposition~2.5]{KLS}, so we give a variational proof. Suppose that $\nu<0$. By Proposition~\ref{prop:weakfield}, we have
    \[\mu_0^{(0)}(b,\nu)<b \mbox{ for $0<b<|\nu|\,$}.\]
    In the sequel, we would like to prove the lower bound
        \[\mu_1^{(0)}(b,\nu)\geq  (3+\nu)b\,.\]
    Towards that goal, we introduce the self-adjoint operator in $L^2(\R_+,r\dd r)$, 
    \[\mathcal L^*=-\frac{\dd^2}{\dd r^2}-\frac1r\frac{\dd}{\dd r}+\frac{r^2}{2}\,,\]
    with domain 
    \[ \mathrm{Dom}(\cL^*)=\{u\in L^2(\R_+,r\dd r) \mbox{ s.t. }\cL^*u\in L^2(\R_+,r\dd r)\mbox{ and } u'(0)=0\}\,.\]
    Note that $\cL^*$ is the radial part of the Landau Hamiltonian with magnetic field $2$, and its  spectrum  consists of the Landau levels $2,6,\cdots$.
    
    The trivial lower bound
    \[w_{b,\nu}^{(0)}(r)\geq r^2-\frac{1}{4r^2}+2\nu\,,\]
    and the min-max principle yield that
    \[\lambda_1(\mathcal S_{0,\nu}^{(0)})\geq \lambda_1(\mathcal S_*)+2\nu\,,\]
    where $\mathcal S_*$ is the self-adjoint operator in $L^2(\R_+)$ defined by
    \[\begin{gathered}
        \mathcal S_*=-\frac{\dd^2}{\dd r^2}+r^2-\frac{1}{4r^2}\,,\\
        \mathrm{Dom}(\mathcal S_*)=\{f\in L^2(\R_+)\colon r^{-1/2}f\in \mathrm{Dom}(\cL^*)\}\,,
    \end{gathered}\]
    which is unitarily equivalent to $\mathcal L_*$. Knowing that $\cL^{(0)}$ is unitarily equivalent to $(b/2)\mathcal S_{b,\nu}^{(0)}$, we get
    \[\lambda_1(\mathcal L^{(0)})\geq \frac{b}{2}\lambda_1(\mathcal S_*)+\nu b\geq (3+\nu)b\,.\]
\end{proof}
\subsection{Finishing the proof of Theorem~\ref{thm:weakfield}}~\\
Consider $\nu\in(-1/2,1/2]$ and an integer $m\geq 0$. If $\nu\geq 0$, we restrict ourselves   to $m\geq 1$ as explained in Remark~\ref{rem:weakfield-m}.

Consider the quasi-mode
\[\Psi(r)=\chi(r)f(r)\,,\]
where
\[\begin{aligned}
    \chi(r)&=1+\frac{m-\nu-b/2}{m-\nu+b/2}r^{-2(m-\nu)}\,,\\
    f(r)&=r^{m-\nu}\ee^{-br^2/4}\,.
\end{aligned}\]
Note that $\Psi'(0)=0$ and $\Psi$ belong to the domain of $\cL^{(m)}_{b,\nu}$. By a straightforward computation as in \cite[Eq. (2.10)]{KLS}, we have
\[\|\Psi\|^2=\begin{cases}
\displaystyle    \frac{2^{m-\nu}}{b^{m-\nu+1}}\Gamma(m-\nu+1)+\cO(1/b)&\mbox{if }m\geq 1\,,\medskip\\
\displaystyle\frac{2^{-\nu}\Gamma(1-\nu)}{b^{1-\nu}} +\cO(1/b)&\mbox{if $m= 0$ and $\nu<0$\,}.
\end{cases}\]
Since $\cL^{(m)}f=bf$ and $\chi''+\frac{1+2(m-\nu)}{r}\chi'=0\,$, we get
\[ (\cL^{(m)}-b)\psi=br\chi'f\,.\]
Consequently,
\[\|(\cL^{(m)}-b)\psi\|=\begin{cases}
    \cO(b^{3/4})&\mbox{if }m\geq 1\,,\\
    \cO(b^{\frac{1-\nu}{2}})&\mbox{if $m= 0$ and $\nu<0$\,}.
\end{cases}\]
Another straightforward computation as in \cite[(2.13)]{KLS} yields
\[\langle (\cL^{(m)}-b)\Psi,\Psi\rangle=
\begin{cases}
    -2(m-\nu)+\cO(\sqrt{b})&\mbox{if }m\geq 1\,,\\
    2\nu+2^{1+\nu}b^{-\nu}\Gamma(1+\nu)+\cO(b)&\mbox{if $m= 0$ and $\nu<0\,$}.
\end{cases}\]
Finally, we apply Temple's inequality
\[\eta-\frac{\epsilon^2}{\beta-\eta}\leq \mu_0^{(m)}(b,\nu,0)-b\leq \eta\,,\]
where
\[\eta=\frac{\langle (\cL^{(m)}-b)\Psi,\Psi\rangle}{\|\Psi\|^2},\quad\epsilon^2=\frac{\|(\cL^{(m)}-b)\Psi\|^2}{\|\Psi\|^2}-\eta^2,\quad \beta=(2+\nu)b\,.\]
Noting that
\[\eta=\begin{cases}
    \displaystyle\frac{-(m-\nu)b^{m-\nu+1}}{2^{m-\nu-1}\Gamma(m-\nu+1)}+\cO(b^{m-\nu+\frac{3}{2}})&\mbox{if }m\geq 1\,,\\
    \displaystyle\frac{2\nu b^{1-\nu}}{2^{-\nu}\Gamma(1-\nu)}+\cO(b^{1-2\nu})&\mbox{if $m= 0$ and $\nu<0\,$},
\end{cases}\]
and
\[\epsilon^2=\begin{cases}\cO(b^{m-\nu+\frac32}),&\mbox{if }m\geq 1\,,\\
\cO(b)&\mbox{if $m= 0$ and $\nu<0$\,},
\end{cases}\]
we get
\[\mu_0^{(m)}(b,\nu)-b=\begin{cases}
   \displaystyle -\frac{b^{m-\nu+1}}{2^{m-\nu-1}\Gamma(m-\nu)}+\cO(b^{m-\nu+\frac32})&\mbox{if }m\geq 1\,,\medskip\\
\displaystyle\frac{2\nu b^{1-\nu}}{2^{-\nu}\Gamma(1-\nu)}+\cO(b^{1-2\nu})&\mbox{if $m= 0$ and $\nu<0\,$}.
\end{cases} \]To finish the proof of Theorem~\ref{thm:weakfield}, it remains to use Propositions~\ref{prop:weakfield} and~\ref{prop:ordering},  and the property of the Gamma function: $\Gamma(1-\nu)=-\nu\Gamma(-\nu)$. 

\begin{remark}\label{rem:mult}
A closer look at the proof shows that,  for $b$ small, there is a unique $m_*(\nu)$ such that $\mu_0^{(m_*)}(b,\nu)=\inf_{m\in\Z}\mu_0^{(m)}(b,\nu)$. Consequently, the ground state energy of the Neumann magnetic Laplacian is a simple eigenvalue. 
\end{remark}


\subsection{An alternative approach via special functions}\label{sec:special functions}~\\
In this subsection, we present an alternative proof of Theorem~\ref{thm:weakfield} that avoids the construction of a quasi-mode.  This proof is based solely on the asymptotics of the confluent hypergeometric function of the second kind $U(a,c,z)$ as $z \to 0$ (as in \cite{ESV, HN}).
For clarity of exposition, we only provide the main ideas and refer the reader to \cite[Subsections~4.3 and~5.2]{HN}, where the weak magnetic field limit is analyzed for the Dirichlet-to-Neumann operator.

The proof is again based on the analysis of the lowest eigenvalues $\mu^{(m)}_0(b,\nu)$ of the fiber operator $\cL^{(m)}$ in \eqref{eq:def-Hm}.
Recall that this fiber operator arises naturally when solving by   separation of variables the eigenvalue problem 
\[
{\mathcal L}\, v = \lambda v \quad \text{in } \Omega\,,
\]
with Neumann boundary conditions on $\partial \Omega$. More precisely, working in polar coordinates $(r,\theta)$ and using the Fourier expansion
\[
v(r,\theta)= \sum_{m\in \Z} v_m(r)\, e^{im\theta},
\]
reduces the problem to the following ODE:
\begin{equation}\label{polarequationsAB}
	\begin{cases}
		- v_m''(r) - \dfrac{v_m'(r)}{r}
		+ \Bigl(br-\dfrac{m-\nu}{r}\Bigr)^2 v_m(r) = \lambda v_m(r)\,,
		& r > 1\,, \\[0.4cm]
		v_m'(1) = 0\,. &
	\end{cases}
\end{equation}

\noindent
For $b>0$, the bounded solutions at infinity are explicitly expressed in terms of the confluent hypergeometric function of the second kind $U(a,c,z)$ (see \cite[Ch.~13]{DLMF} or \cite[Ch.~VI]{MOS1966}). One finds
\begin{equation}\label{eigenfunctions}
	v_m(r) = c_{m,\nu}\, e^{-br^2/2}\, r^{m-\nu}\,
	U\!\left(\tfrac{1}{2}-\tfrac{\lambda}{2b},\, m-\nu+1,\, \tfrac{b}{2}r^2\right),
\end{equation}
where $c_{m,\nu}$ is an arbitrary constant.  

We recall the following integral representation (see \cite[p.~277]{MOS1966}):
\begin{equation}\label{eq:integralrepU}
	U(a,c,z)= \frac{1}{\Gamma(a)}
	\int_0^{+\infty} e^{-zt}\, t^{a-1}\,(1+t)^{c-a-1}\, dt\,,
	\quad \Re a >0\,,\ \Re z >0\,.
\end{equation}

\vspace{0.2cm}
\noindent
Differentiating with respect to $z$ yields
\begin{equation}\label{derivU}
	U'(a,c,z) := -a\, U(a+1,c+1,z)\,.
\end{equation}

\vspace{0.2cm}\noindent
The Neumann condition $v_m'(1)=0$, combined with \eqref{eigenfunctions} and \eqref{derivU}, implies that the eigenvalues $\lambda$ of ${\mathcal L}^{(m)}$ satisfy the implicit equation
\begin{equation}\label{implicit}
	\begin{split}
		\bigl(m-\nu - \tfrac{b}{2}\bigr)\,
		&U\!\left(\tfrac{1}{2}-\tfrac{\lambda}{2b},\, m-\nu+1,\, \tfrac{b}{2}\right) \\
		&+ \bigl(\tfrac{\lambda}{2}-\tfrac{b}{2}\bigr)\,
		U\!\left(\tfrac{3}{2}-\tfrac{\lambda}{2b},\, m-\nu+2,\, \tfrac{b}{2}\right) = 0\,.
	\end{split}
\end{equation}
Recall (see Corollary~\ref{corol:weakfield}) that for $0<b<|\nu|$, the lowest eigenvalue of the magnetic Neumann Laplacian satisfies
\[
\mu(b,\nu,0)= \inf_{m\geq 0} \mu^{(m)}_0(b,\nu)\,,
\]
while for $\nu\geq 0$\,,
\[
\mu(b,\nu,0)= \inf_{m\geq 1} \mu^{(m)}_0(b,\nu)\,.
\]

\vspace{0.2cm}\noindent
We now distinguish three cases according to the value of $\nu$.

\paragraph{\it{Case $\nu \in (0,\tfrac12)$.}}

Figure~1 suggests that for small $b$,
\[
\mu(b,\nu,0)= \mu_0^{(1)}(b,\nu)\,.
\]
and in particular, the corresponding ground state is \emph{not} radial. This motivates us to study the asymptotics of $\mu_0^{(1)}(b,\nu)$ first.
The implicit equation for $ \mu_0^{(1)}(b,\nu)$ is
\begin{equation}\label{implicitpositif}
	\begin{split}
		\bigl(1-\nu - \tfrac{b}{2}\bigr)\,
		&U\!\left(\tfrac{1}{2}-\tfrac{\mu_0^{(1)}(b,\nu)}{2b},\, 2-\nu,\, \tfrac{b}{2}\right) \\
		&+ \bigl(\tfrac{\mu_0^{(1)}(b,\nu)}{2}-\tfrac{b}{2}\bigr)\,
		U\!\left(\tfrac{3}{2}-\tfrac{\mu_0^{(1)}(b,\nu)}{2b},\, 3-\nu,\, \tfrac{b}{2}\right) = 0\,.
	\end{split}
\end{equation}
To analyze $\mu_0^{(1)}(b,\nu)$ as $b\to 0^+$, we use the expansions (see \cite[p.~288]{MOS1966} and \cite{SoSo}):
\begin{align} 
	U(a,c,z) &= \frac{\Gamma(c-1)}{\Gamma(a)}\, z^{1-c}
	+ \frac{\Gamma(1-c)}{\Gamma(a-c+1)}
	+ \mathcal O(z^{2-c})\,, 
	& \quad 1<c<2\,, \label{eq:Usmallz1}\\[0.2cm]
	U(a,c,z) &= \frac{\Gamma(c-1)}{\Gamma(a)}\, z^{1-c}
	+ \mathcal O(z^{2-c})\,,
	& \quad c>2\,. \label{eq:Usmallz2}
\end{align}
These asymptotics are uniform with respect to $a$  in a neighborhood of $0$.  
Thanks to Proposition~\ref{prop:weakfield-m}, we look for an expansion of the form
\[
\mu_0^{(1)}(b,\nu) = b- A b^{\epsilon} + o( b^{\epsilon})\,,
\]
where $\epsilon>1$ and $A$ are constants to be determined.  
Substituting \eqref{eq:Usmallz1}--\eqref{eq:Usmallz2} into \eqref{implicitpositif} and using $\Gamma(z+1)=z\Gamma(z)$, one finds after straightforward but tedious calculations that
\[
\epsilon =2-\nu,
\qquad
A = \frac{2^\nu}{\Gamma(1-\nu)}.
\]
which is consistent with Theorem~\ref{thm:weakfield}(4) in the case $k=0$.
\vspace{0.5cm}
\paragraph{\it{Case $\nu \in (-\tfrac12,0)$.}}

The situation is similar, except that here  Figure~1 suggests that the lowest dispersion curve is $\mu(b,\nu,0)= \mu_0^{(0)}(b,\nu)$, and  in particular that the ground state is radial.
Repeating the same arguments as above, we find an asymptotics of $\mu_0^{(0)}(b,\nu)$ consistent with Theorem~\ref{thm:weakfield}(3).

\vspace{0.5cm}
\paragraph{\it{Case $\nu=0$.}}

A slightly different analysis is required since $c$ is an integer in this case.  
The small--$z$ asymptotics of $U(a,c,z)$ read (see \cite[Eq.~(13.2.9), (13.2.16)]{DLMF})
\begin{align}\label{asymptotUzero}
	U(a,2,z) &= \frac{1}{\Gamma(a)}\,z^{-1}
	+ \frac{\log z + \psi(a) + 2\gamma - 1}{\Gamma(a-1)}
	+ \mathcal O\!\bigl(z\log z\bigr), \\[0.3cm]
	U(a,3,z) &= \frac{1}{\Gamma(a)}\,z^{-2}
	+ \mathcal O\!\bigl(z^{-1}\bigr) 	+ \mathcal O\!\big( \psi(a)\bigr),
\end{align}
uniformly with respect to $a$ for $a$ in a neighborhood of $0$, where $\psi(a)=\Gamma'(a)/\Gamma(a)$ is the digamma function and $\gamma$ is the Euler--Mascheroni constant.  
As $a\to 0$, one has (see \cite[5.7(ii)]{DLMF})
\[
\psi(a) = -\frac{1}{a} - \gamma + \mathcal O(a)\,.
\]
In this case,  we still expect that $\mu(b,0,0)= \mu_0^{(1)}(b,0)$ (see Figure~2). 
Following the same procedure as above and using \eqref{asymptotUzero}, one eventually recovers, after somewhat lengthy computations, an asymptotics of $\mu_0^{(1)}(b,0)$ consistent with  Theorem~\ref{thm:weakfield}(4).

 The above argument applies to each dispersion curve $\mu^{(m)}_0(b,\nu)$ for any fixed $m\geq 1$, yielding a two-term asymptotic expansion as $b\to0^+$. Consequently, for sufficiently small $b$, we have
\[
\mu_0^{(0)}(b,\nu)<\mu_0^{(1)}(b,\nu)\quad (\nu<0).
\]
Moreover, by Proposition~\ref{prop:ordering}, we have the ordering 
\[\mu_0^{(1)}(b,\nu)<\mu_0^{(2)}(b,\nu)<\cdots\]
for all $\nu\in(-1/2,1/2]$ and for $0<b<2(2-\nu)-\sqrt{8(2-\nu)+1}$. This allows us to finish the proof of Theorem~\ref{thm:weakfield}.

\begin{figure}
	\begin{center}
		\includegraphics[width=0.49\textwidth]{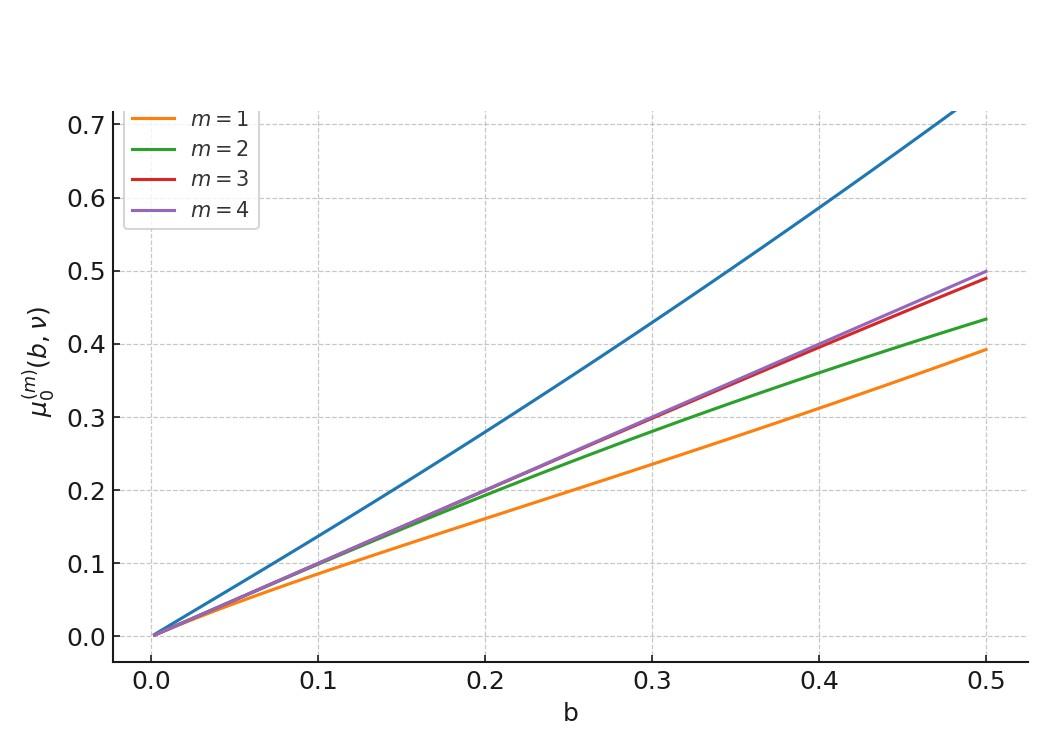}
		\includegraphics[width=0.49\textwidth]{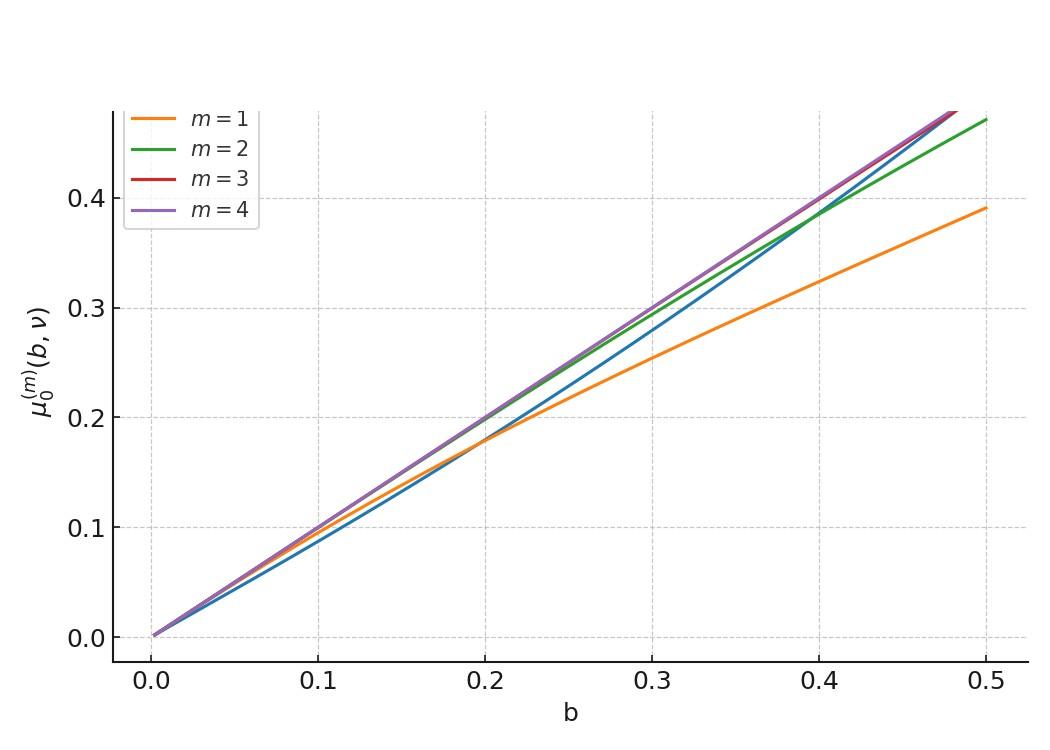}
	\end{center}
	\caption{Dispersion curves $\mu^{(m)}_0(b,\nu)$ for $\nu = \tfrac14$ (left) and $\nu=-\tfrac14$ (right).}
\end{figure}

\begin{figure}
	\begin{center}
		\includegraphics[width=0.49\textwidth]{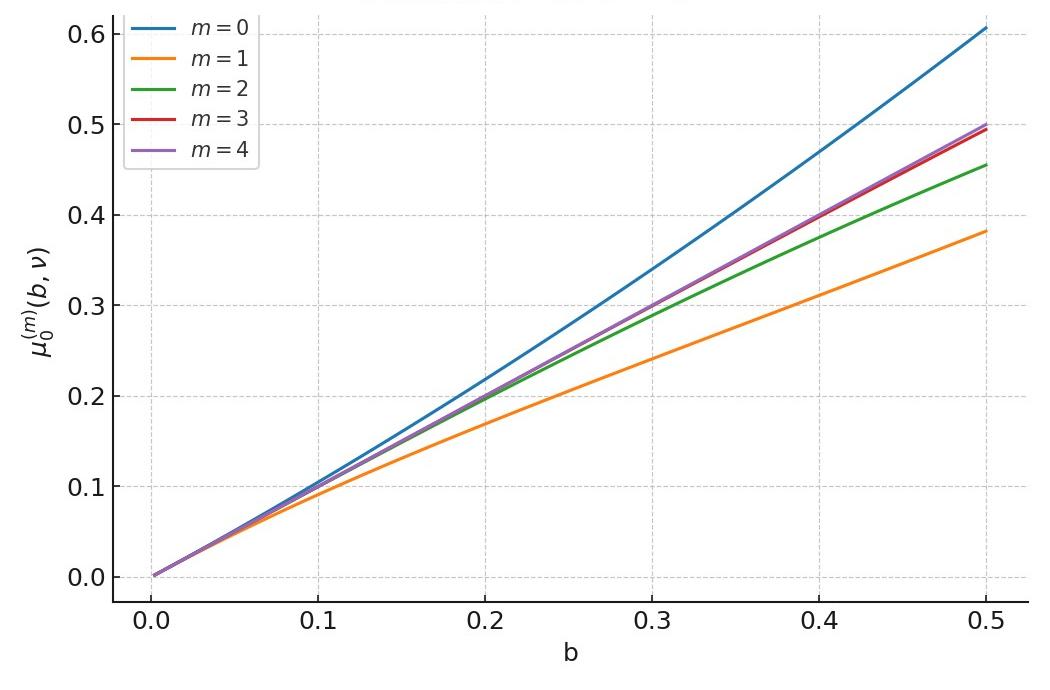}
	\end{center}
	\caption{Dispersion curves $\mu^{(m)}_0(b,\nu)$ for $\nu =0$.}
\end{figure}


\subsection*{Acknowledgments} The authors would like to thank the referee for the helpful comments that improved the presentation of the paper.


\section*{Declarations and statements}

{\bf Research funding}. A.K. was partially supported by a startup fund at AUB (grant no. 513125). F.N. was partially supported by   the French GDR Dynqua.\\

{\bf Conflict of Interest}. The authors declare that they have no competing interests regarding the publication of this paper.\\

{\bf Availability of data}. There are no data associated with the research in this paper.\\

{\bf Author contributions}. All authors contributed equally to the study, read, and approved the final version of the submitted manuscript.


\begin{thebibliography}{100}
%
\bibitem{CG} A. Chaigneau, D. S. Grebenkov.
	\newblock The Steklov problem for exterior domains: asymptotic behavior and applications.  
	\newblock 
    J.  Math. Phys.  {\bf 66}  (2025), art. no. 061502. 
%
\bibitem{CGHP} T. Chakradhar, K. Gittins, G. Habib, N. Peyerimhoff. \newblock 
A note on the magnetic Steklov operator on functions. \newblock	 
Mathematika 71: e70037  (2025).

\bibitem{DLMF}
{\it NIST Digital Library of Mathematical Functions}.
\newblock https://dlmf.nist.gov/, Release 1.2.0 of 2024-03-15.
\newblock F.~W.~J. Olver, A.~B. {Olde Daalhuis}, D.~W. Lozier, B.~I. Schneider,
R.~F. Boisvert, C.~W. Clark, B.~R. Miller, B.~V. Saunders, H.~S. Cohl, and
M.~A. McClain, eds. 


%
\bibitem{ESV}  P. Exner, P. $\check{\rm S}$t'ov\'i$\check{\rm c}$ek, P. Vyt$\check{\rm r}$as. \newblock Generalized boundary conditions for the 
Aharonov-Bohm
effect combined with a homogeneous magnetic field. \newblock  
J. Math. Phys. {\bf 43} (2002), 2151--2168. 

%
\bibitem{FTRV} R. Fahs, L. Treust, N. Raymond, S. V\~u Ng\d{o}c.
\newblock Boundary states of the Robin magnetic Laplacian. 
\newblock Doc. Math. {\bf 29} (2024), 1157--1200.
%
\bibitem{GKS}  M. Goffeng, A. Kachmar, M.P. Sundqvist. \newblock Clusters of eigenvalues for the magnetic Laplacian
with Robin condition. \newblock J. Math. Phys. {\bf 57} (2016), art. no. 063510.
%
\bibitem{FH-b} S. Fournais, B. Helffer.
\newblock Spectral Methods in Surface Superconductivity. 
\newblock Progress in Nonlinear Differential Equations and Their Applications 77. Basel: Birkh\"auser (2010).
%
\bibitem{FH-cmp} S. Fournais,  B. Helffer. \newblock On the third critical field in Ginzburg-Landau theory. \newblock Comm. Math. Phys. {\bf 266} (2006), 153--196.
%
\bibitem{FM} S. Fournais, \newblock L. Morin. Magnetic tunneling between disc-shaped obstacles. \newblock Comm. Math. Phys. {\bf 406} (2025), art. no. 114.
%
\bibitem{FS} 
S. Fournais, M.P. Sundqvist. \newblock Lack of Diamagnetism and the Little-Parks Effect. \newblock Comm. Math. Phys. {\bf 337} (2015), 191--224.
%

\bibitem{HKN} B. Helffer, A. Kachmar, F. Nicoleau.  
\newblock Asymptotics for the magnetic Dirichlet-to-Neumann eigenvalues in general domains.
%
\newblock J. Func. Anal.
{\bf 290} (2026),  art. no. 111449.
%

\bibitem{HL}
B. Helffer, C. L\'ena. \newblock
Eigenvalues of the Neumann magnetic Laplacian in the unit disk. \newblock 
J. Math. Phys. {\bf 66} (2025), art. no.  081513.
%
\bibitem{HN1} B. Helffer, F. Nicoleau. \newblock
On the magnetic Dirichlet to Neumann operator on the disk: strong diamagnetism and strong magnetic field limit. \newblock J. Geom. Anal. {\bf 35} (2025), art. no. 178.

\bibitem{HN} B. Helffer, F. Nicoleau. \newblock  On the magnetic Dirichlet to Neumann operator on the exterior of the disk -- diamagnetism, weak-magnetic field limit and flux effects. \newblock  J. Math. Pures Appl. (9) {\bf 205} (2026), art. no. 103799.
%
\bibitem{K-jmp} A. Kachmar.  
\newblock On the ground state energy for a magnetic Schr\"odinger operator and the effect of the De Gennes boundary condition. 
\newblock J. Math. Phys. {\bf 47} (2006), art. no. 072106.

\bibitem{KLS}  A. Kachmar, V. Lotoreichik,  M.P. Sundqvist. \newblock
On the Laplace operator with a weak magnetic field in exterior domains. \newblock
Anal. Math. Phys. {\bf 15} (2025), art. no. 5.








\bibitem{MOS1966} W. Magnus, F. Oberhettinger,  R.P. Soni.
\newblock Formulas and theorems for the special functions of mathematical physics,  3rd enlarged ed, 
\newblock Grundlehren der Mathematischen Wissenschaften,
\newblock Volume 52,
\newblock Springer,  (1965).

\bibitem{Shen} Z. Shen.  
\newblock The magnetic Laplacian with a higher-order vanishing magnetic field in a bounded domain.  
\newblock arXiv:2505.03690 (2025).

\bibitem{SoSo} S. Soojin Son. 
\newblock Spectral Problems on Triangles and Discs:
Extremizers and Ground States. 
PhD thesis. 2014. \\ url:
https://www.ideals.illinois.edu/items/49400.

%
\end{thebibliography}
\end{document}